\documentclass[11pt]{amsart}

\usepackage{a4wide,amsmath,amsfonts,amssymb,amsthm,mathrsfs,bbm}
\usepackage[hyperfootnotes=false]{hyperref}

\allowdisplaybreaks[2]

\numberwithin{equation}{section}

\newtheorem{theorem}{Theorem}[section]
\newtheorem{lemma}[theorem]{Lemma}

\newtheorem{remark}[theorem]{Remark}
\newtheorem{proposition}[theorem]{Proposition}
\newtheorem{definition}[theorem]{Definition}
\newtheorem{assumption}[theorem]{Assumption}

\renewcommand{\d}{\mathrm{d}}
\newcommand{\dd}{\,\mathrm{d}}
\newcommand{\R}{\mathbb{R}}
\newcommand{\bR}{\mathbb{R}}
\newcommand{\N}{\mathbb{N}}
\newcommand{\bN}{\mathbb{N}}
\newcommand{\bE}{\mathbb{E}}
\renewcommand{\P}{\mathbb{P}}
\renewcommand{\epsilon}{\varepsilon}
\newcommand{\bP}{\mathbb{P}}
\def\L{{\mathcal{L}}}

\title[Weak solutions to distribution-dependent SVEs]{Weak solutions to distribution-dependent stochastic Volterra equations}

\author[Bergerhausen]{Martin Bergerhausen}
\address{Martin Bergerhausen, University of Mannheim, Germany}
\email{martin.bergerhausen@uni-mannheim.de}

\author[Pr{\"o}mel]{David J. Pr{\"o}mel}
\address{David J. Pr{\"o}mel, University of Mannheim, Germany}
\email{proemel@uni-mannheim.de}

\date{\today}

\begin{document}
 
\begin{abstract}
  We prove the existence of weak solutions for distribution-dependent stochastic Volterra equations under linear growth and continuity conditions on the coefficients and mild regularity assumptions on the kernels, including singular kernels. To this end, we formulate an associated local martingale problem and establish its connection with weak solutions. Moreover, we derive continuity and integrability properties of the solutions.
\end{abstract}

\maketitle



% Classification
% --------------
%  60H20 - Stochastic analysis (Stochastic integral equations)
%  45D05 - Volterra integral equations (Volterra integral equations)

%\tableofcontents

\section{Introduction}

Distribution-dependent stochastic Volterra equations provide a natural framework that combines two important generalizations of classical stochastic differential equations: stochastic Volterra equations, which incorporate memory effects through kernels, and distribution-dependent (mean-field or McKean--Vlasov) stochastic differential equations, where the coefficients depend on the law of the solution. By bringing together these two features, distribution-dependent stochastic Volterra equations offer a flexible modeling framework capable of capturing both non-Markovian dynamics and distributional interactions.

Stochastic Volterra equations, first studied by Berger and Mizel \cite{Berger1980a,Berger1980b}, extend classical stochastic differential equations by allowing the dynamics to depend on the past through a kernel. While they are a classical object in probability theory with a wide range of applications, they have received renewed attention in recent years, notably due to their role in rough volatility models in mathematical finance; see, for instance, \cite{Jaisson2016,ElEuch2019}. In parallel, distribution-dependent stochastic differential equations, also known as mean-field or McKean--Vlasov stochastic differential equations, describe interacting particle systems in which the dynamics of each component depend on the distribution of the entire system. Their study dates back to the seminal works of Kac~\cite{Kac56}, McKean~\cite{McKean66} and Vlasov~\cite{Vlasov1968}, and they continue to form a very active area of research with numerous applications in fields such as physics, finance and data science; see, e.g., \cite{Carmona2018,Carmona2018b,Chaintron2022,Chaintron2022b} for comprehensive overviews.

While stochastic Volterra equations and distribution-dependent stochastic differential equations have been extensively studied individually, their combination constitutes a natural generalization from both a mathematical and modeling perspective, allowing one to incorporate memory effects and mean-field interactions within a single framework. This has motivated a growing body of research on distribution-dependent stochastic Volterra equations; see, e.g., \cite{Shi2013,Promel_MF,Jie2024,Liu2023,Kalinin2026}. However, this increased generality also introduces significant mathematical challenges, as the resulting dynamics are typically non-Markovian and need not be semimartingales, thereby limiting the applicability of classical tools from stochastic analysis such as It\^o's formula, Gronwall-type arguments, and Burkholder--Davis--Gundy inequalities.

The aim of this paper is to study the existence of weak solutions to the multi-dimensional distribution-dependent stochastic Volterra equation (distribution-dependent SVE)
\begin{equation}\label{eq:SVE intro}
  X_t = X_0 + \int_0^t K_b(s,t)\, b(s,X_s,\mathcal L(X_s))\dd s + \int_0^t K_\sigma(s,t)\, \sigma(s,X_s,\mathcal L(X_s))\dd B_s ,
\end{equation}
where the time-inhomogeneous drift and diffusion coefficients $b,\sigma$ depend on the state and the law of the solution, the kernels $K_b,K_\sigma$ are measurable functions, and $B=(B_t)_{t\in [0,T]}$ is multi-dimensional Brownian motion. Let us remark that we allow for kernels that may be singular and does not require a convolution structure.

As first contribution, we introduce local martingale problems associated to the distribution-dependent SVE~\eqref{eq:SVE intro}. The presented distribution-dependent Volterra local martingale problems naturally extend the classical martingale problems for (McKean--Vlasov) stochastic differential equations, cf. e.g. \cite{Stroock1979,Funaki1984}, as well as the martingale problems for stochastic Volterra quation \cite{AbiJaber2021,Promel2023_weak}. Under fairly mild integrability conditions on the kernels and a continuity assumption on the coefficients, we prove that the existence of weak solutions to distribution-dependent SVE~\eqref{eq:SVE intro} is equivalent to the solvability of the associated Volterra local martingale problems, see Proposition~\ref{prop:equivalence}.

As a second and main contribution, we establish the existence of weak solutions under linear growth and continuity assumptions on the coefficients, together with mild regularity and integrability conditions on the kernels; see Theorem~\ref{thm: weak solution to mean-field SVEs}. In contrast to the approximation techniques employed in \cite{AbiJaber2021,Promel2023_weak,AbiJaber2025} for classical stochastic Volterra equations, we introduce an Euler-type approximation scheme tailored to the distribution-dependent SVE~\eqref{eq:SVE intro}, which appears to be more suitable in view of the law dependence of the coefficients. While our approach follows the classical strategy of combining approximation methods with compactness arguments, the Volterra structure gives rise to substantial additional difficulties. We first derive uniform moment estimates for the approximating processes, which ensure tightness in the space of continuous paths. This is achieved via suitable increment estimates and an application of Kolmogorov's criterion. By Prokhorov's theorem, we obtain the existence of a convergent subsequence, and a Skorokhod representation argument yields almost sure convergence on a suitable probability space. Finally, we identify the limit by showing that it solves the Volterra local martingale problem introduced earlier. As a by-product, the moment estimates also yield continuity and integrability properties of the resulting solutions.

Let us conclude this introduction with a brief discussion of related results in the existing literature. For classical distribution-dependent stochastic differential equations, the martingale problem formulation and the existence of weak solutions are well studied and by now classical topics; see, e.g., \cite{Stroock1979,Funaki1984} and the references therein.

The existence of weak solutions to stochastic Volterra equations has been established through a variety of approaches extending classical SDE techniques to the non-Markovian setting. For convolution-type SVEs, weak existence and stability are obtained via martingale problem formulations and $L^p$-based weak convergence methods \cite{AbiJaber2021}, see also \cite{Mytnik2015}. Additional existence results for affine and polynomial Volterra frameworks under structural and invariance conditions are provided in \cite{AbiJaber2019,AbiJaber2024}. For general kernels, a Volterra martingale problem formulation and corresponding weak existence results for time-inhomogeneous coefficients are established in \cite{Promel2023}. These results have been further extended to singular and non-convolution kernels via approximation and stability methods \cite{AbiJaber2025}, while \cite{Hamaguchi2025} proves weak existence and uniqueness in law for SVE with completely monotone kernels. However, to the best of our knowledge, results on the existence of weak solutions for distribution-dependent stochastic Volterra equations appear to be absent from the current literature.

The existence of strong solutions of distribution-dependent stochastic Volterra equations has been established under various structural assumptions, extending techniques from both stochastic Volterra equations and McKean--Vlasov SDEs. Early results in \cite{Shi2013} prove existence and uniqueness of strong solutions for linear distribution-dependent SVEs. For general distribution-dependent SVEs, strong existence and pathwise uniqueness are obtained under Lipschitz coefficients and possibly singular kernels using fixed-point and Yamada--Watanabe arguments \cite{Promel_MF}. Independently, similar existence and uniqueness results have been extended to settings with H{\"o}lder continuous diffusion coefficients \cite{Jie2024}, as well as to distribution-dependent SVEs with singular kernels under Lipschitz conditions on the coefficients via Volterra-type Gronwall inequalities \cite{Liu2023}. More recently, in \cite{Kalinin2026}, strong solutions have also been derived in Banach space settings with distribution-dependent coefficients and singular kernels based on refined measurability and moment estimates.

\medskip
\noindent \textbf{Organization of the paper:} In Section~\ref{sec:VMP} we introduce a local martingale problem associated to distribution-dependent SVEs. The existence of weak solutions to distribution-dependent SVEs is provided in Section~\ref{sec:construction}. In Section~\ref{sec:rest}, we show continuity and integrability properties of solutions. The Appendix~\ref{sec: appendix} contains an auxiliary result regarding the convergence of Volterra-type integrals.

\medskip
\noindent\textbf{Acknowledgments:} D. J. Pr\"omel gratefully acknowledges his affiliation with the Department of Mathematics at King’s College London, United Kingdom.

\section{Weak solutions and Volterra local martingale problems}\label{sec:VMP}

Let $T\in (0,\infty)$, $d,m\in\N$, and let $(\Omega,\mathcal{F},(\mathcal{F}_t)_{t\in [0,T]},\mathbb{P})$ be a filtered probability space, which satisfies the usual conditions of completeness and right-continuity. Suppose $B=(B_t)_{t\in [0,T]}$ is an $m$-dimensional Brownian motion with respect to $(\mathcal{F}_t)_{t\in [0,T]}$. The law of a random variable~$X$ is denoted by $\mathcal{L}(X)$ and, if $(E,d)$ is a complete separable metric space with distance~$d$, for any $p \geq 1$ we denote by $\mathcal{P}_p(E)$ the space of probability measures with finite $p$-th moments. For $\rho,\tilde{\rho}\in\mathcal{P}_p(E)$, the $p$-Wasserstein distance $W_p(\rho,\tilde{\rho})$ is defined by
\begin{equation*}
  W_p(\rho,\tilde{\rho}):= \inf_{\pi \in \Pi(\rho,\tilde{\rho})} \Big[ \int_{E \times E} d(x,y)^p \,\pi(\d x,\d y)\Big]^{1/p},
\end{equation*}
where $\Pi(\rho,\tilde{\rho})$ denotes the set of probability measures on $E\times E$ with marginals $\rho$ and $\tilde{\rho}$, see \cite[Chapter~5]{Carmona2018} for more details. The space $\R^d$ is equipped with the Euclidean norm~$| \cdot |$ and we set $\Delta_T:=\lbrace (s,t)\in [0,T]\times [0,T]\colon \, 0\leq s\leq t\leq T \rbrace$. For a measure space $\mathcal{X}$, a Banach spaces $\mathcal{Y}$ and $p\geq 1$, we use the notations $L^p(\mathcal{X};\mathcal{Y})$ for the space of all $\mathcal{Y}$-valued, measurable, $p$-integrable functions on $\mathcal{X}$, and we write $L^p(\mathcal{X}):=L^p(\mathcal{X};\R)$.

\medskip

Let $b\colon[0,T]\times\R^d\times\mathcal{P}_{\eta}(\R^d)\to\R^d$, $\sigma\colon[0,T]\times\R^d\times\mathcal{P}_{\eta}(\R^{d})\to\R^{d\times m}$ for some $\eta \in \bN$ and the kernels $K_b, K_\sigma\colon \Delta_T\to \R$ be measurable functions. We consider the $d$-dimensional distribution-dependent stochastic Volterra equation (distribution-dependent SVE)
\begin{equation}\label{eq:MVSVE}
  X_t = X_0 +\int_0^t K_{b}(s,t) b(s,X_s,\mathcal{L}(X_s))\dd s+\int_0^t K_{\sigma}(s,t)\sigma(s,X_s,\mathcal{L}(X_s))\dd B_s,\quad t\in [0,T],
\end{equation}
with initial distribution $\L(X_0)=\mu_0 \in \mathcal{P}_{p}(\R^d)$, that is, $X_0$ is a $d$-dimensional, $\mathcal{F}_0$-measurable random variable with distribution $\mu_0$ and with finite $p$-th moments, $p \geq \eta$, which is independent of $B$. The integral $\int_0^t K_{\sigma}(s,t)\sigma(s,X_s,\mathcal{L}(X_s))\dd B_s$ is defined as a stochastic It{\^o} integral.

\medskip

Let us recall the concepts of well-posedness, strong solutions and pathwise uniqueness. An $(\mathcal{F}_t)$-progressively measurable stochastic process $(X_t)_{t\in [0,T]}$ in $L^p(\Omega\times [0,T];\R^d)$, on the given probability space $(\Omega,\mathcal{F},(\mathcal{F}_t)_{t\in[0,T]},\mathbb{P})$, is called a \textit{(strong) $L^p$-solution} of the distribution-dependent SVE~\eqref{eq:MVSVE} if
\begin{equation*}
  \int_0^t (|K_b(s,t)b(s,X_s,\mathcal{L}(X_s))|+|K_\sigma(s,t)\sigma(s,X_s,\mathcal{L}(X_s))|^2 )\dd s<\infty \quad \text{for all }t\in[0,T],
\end{equation*}
and the integral equation~\eqref{eq:MVSVE} holds $\mathbb{P}$-almost surely. We say \textit{pathwise uniqueness in} $L^p$ holds for the distribution-dependent SVE~\eqref{eq:MVSVE} if $\mathbb{P}(X_t=\tilde{X}_t, \,\forall t\in [0,T])=1$ for any two $L^p$-solutions $(X_t)_{t\in[0,T]}$ and $(\tilde{X}_t)_{t\in[0,T]}$ of \eqref{eq:MVSVE} defined on the same probability space $(\Omega,\mathcal{F},(\mathcal{F}_t)_{t\in[0,T]},\mathbb{P})$. We say that the distribution-dependent SVE~\eqref{eq:MVSVE} is \textit{well-posed in $L^p$} (or that there exists a \textit{unique $L^p$-solution}) for $p\geq 1$ if there exists a strong $L^p$-solution to \eqref{eq:MVSVE} and pathwise uniqueness in $L^p$ holds.

\medskip

Following to the notion of weak solutions to ordinary stochastic differential equations (see e.g. \cite[Definition~5.3.1]{Karatzas1991}) and to stochastic Volterra equations (see e.g. \cite[Definition~2.1]{Promel2023_weak}), we make the following definition.

\begin{definition}\label{def:weak solution}
  A \textup{weak solution} to \eqref{eq:MVSVE} is a triple $((X,B),(\Omega,\mathcal{F},\mathbb{P}),(\mathcal{F}_t)_{t\in [0,T]})$ such that
  \begin{enumerate}
    \item[(i)] $(\Omega,\mathcal{F},\mathbb{P})$ is a probability space, $(\mathcal{F}_t)_{t\in [0,T]}$ is a filtration of sub-$\sigma$-algebras of $\mathcal{F}$ satisfying the usual conditions,
    \item[(ii)] $X=(X_t)_{t\in [0,T]}\in L^1(\Omega~\times~ [0,T]; \mathbb{R}^d)$ is an $(\mathcal{F}_t)$-progressively measurable stochastic process, $B=(B_t)_{t\in [0,T]}$ is an $m$-dimensional Brownian motion w.r.t. $(\mathcal{F}_t)_{t\in [0,T]}$,
    \item[(iii)] $\int_0^t\big(|K_b(s,t)b(s,X_s,\mathcal{L}(X_s))| + |K_\sigma(s,t)\sigma(s,X_s,\mathcal{L}(X_s))|^2  \big)\dd s<\infty$ $\P$-a.s. for any $t\in[0,T]$,
    \item[(iv)] \eqref{eq:MVSVE} holds for $(X,B)$ on $(\Omega,\mathcal{F},\mathbb{P})$, $\P$-a.s., and
    \item[(v)] $X_0$ is $\mu_0$-distributed.
  \end{enumerate}
\end{definition}

Under suitable assumptions on the coefficients and kernels, the existence of weak solutions to the distribution-dependent SVE~\eqref{eq:MVSVE} can be equivalently formulated in terms of solutions to associated local martingale problems, see Definition~\ref{def: MP for stoch process} and~\ref{def:martproblem_new} below. To that end, we make the following assumption, where $\delta_x$ denotes the Dirac measure of $x\in \bR^d$.

\begin{assumption}\label{ass:short}
  Let $K_b, K_\sigma\colon \Delta_T\to \R$ be measurable functions with $K_b(\cdot,t) \in L^1([0,t])$ and $K_\sigma(\cdot,t)\in L^2([0,t])$ for every $t\in[0,T]$, and let $b \colon [0,T]\times\R^d \times \mathcal{P}_{\eta}(\R^d) \to\R^d$ and $\sigma \colon [0,T]\times\R^d \times \mathcal{P}_{\eta}(\R^d) \to\R^{d\times m}$ be measurable functions fulfilling the linear growth condition
  \begin{equation*}
    |b(t,x,\rho)|+|\sigma(t,x,\rho)| \leq C_{b,\sigma} (1 +|x|+W_\eta(\delta_0, \rho)), \quad  t\in [0,T],\, x\in \R^d,\,\rho \in \mathcal{P}_{\eta}(\R^d),
  \end{equation*}
  for some constant $C_{\mu,\sigma}>0$.
\end{assumption}

Let $ C^2(\R^d)$ be the space of twice continuously differentiable functions $f\colon\R^d\to \R$ and $C_0^2(\R^d)$ be the space of all $f\in C^2(\R^d)$ with compact support. For two stochastic processes $X=(X_t)_{t\in [0,T]}$ and $Z=(Z_t)_{t\in [0,T]}$ with $X_0=Z_0$ on a filtered probability space $(\Omega,\mathcal{F},(\mathcal{F}_t)_{t\in[0,T]},\P)$ satisfying the usual conditions, such that $X\in L^1(\Omega\times [0,T];\R^d)$ is $(\mathcal{F}_t)$-progressively measurable and $Z$ is $(\mathcal{F}_t)$-adapted and continuous, we introduce the process $(\mathcal{M}_t^{f})_{t\in[0,T]}$ by
\begin{equation}\label{def:M_f1}
  \mathcal{M}^{f}_t:= f(Z_t)-\int_0^t \mathcal{A}^{f}(s,X_s,\mathcal{L}(X_s),Z_s)\dd s,\quad t\in [0,T],
\end{equation}
for $f\in C^2(\R^d)$, where $\mathcal{A}^{f}\colon[0,T]\times\R^d\times \mathcal{P}_{\eta}(\R^d) \times \R^d\to \R$ with
\begin{equation}\label{eq:operator A}
  \mathcal{A}^{f}(t,x,\rho,z):=\sum_{i=1}^d b_i(t,x,\rho)\frac{\partial f(z)}{\partial z_i}+\frac{1}{2}\sum_{i,j=1}^d (\sigma\sigma^\top)_{ij}(t,x,\rho) \frac{\partial^2 f(z)}{\partial z_i \partial z_j}.
\end{equation}

As we shall see in the next lemma, assuming that $(\mathcal{M}^f_t)_{t\in [0,T]}$ is a local martingale for all $f\in C_0^2(\R)$ implies that the stochastic process~$Z$ is a semimartingale.

\begin{lemma}\label{lem:Zsemimart}
  Suppose Assumption~\ref{ass:short}. Let $(X_t)_{t\in [0,T]}$ be an $(\mathcal{F}_t)$-progressively measurable process in $L^1(\Omega\times [0,T];\R^d)$ and $(Z_t)_{t\in [0,T]}$ be an $(\mathcal{F}_t)$-adapted and continuous $d$-dimensional  stochastic process on a filtered probability space $(\Omega,\mathcal{F},(\mathcal{F}_t)_{t\in[0,T]},\P)$ with $Z_0=X_0$ $\P$-a.s. satisfying the usual conditions. If $(\mathcal{M}^f_t)_{t\in[0,T]}$ is a local martingale for every $f\in C_0^2(\R^d)$, then we have:
  \begin{enumerate}
    \item[(i)] $(Z_t)_{t\in[0,T]}$ is a semimartingale with characteristics
    \begin{equation*}
      \bigg(\int_0^{\cdot} b(s,X_s, \mathcal{L}(X_s))\dd s , \int_0^{\cdot}(\sigma \sigma^\top)(s,X_s,\mathcal{L}(X_s))\dd s,0\bigg).
    \end{equation*}
    \item[(ii)] There exists a filtered probability space $(\tilde{\Omega},\tilde{\mathcal{F}},(\tilde{\mathcal{F}}_t)_{t\in[0,T]},\tilde{\P})$ satisfying the usual conditions such that $(Z_t)_{t\in [0,T]}$ is a semimartingale on $(\tilde{\Omega},\tilde{\mathcal{F}},(\tilde{\mathcal{F}}_t)_{t\in[0,T]},\tilde{\P})$ and
    \begin{equation*}
      Z_t= X_0+\int_0^t b(s,X_s,\mathcal{L}(X_s))\dd s + \int_0^t \sigma(s,X_s,\mathcal{L}(X_s))\dd B_s,\quad t\in[0,T],
    \end{equation*}
    holds $\tilde{\P}$-a.s., for some $m$-dimensional Brownian motion $(B_t)_{t\in[0,T]}$ on the mentioned filtered probability space.
  \end{enumerate}
\end{lemma}

\begin{proof}
  (i) By \cite[Theorem~II.2.42]{Jacod2003}, in order to prove the assertion, it is sufficient to show that $(\mathcal{M}_t^f)_{t\in[0,T]}$, defined in~\eqref{def:M_f1}, is a local martingale for every bounded function $f\in C^2(\R^d)$.

  Let $f\in C^2(\R^d)$ be a bounded function and define the hitting times
  \begin{equation*}
    \tau_n:=\inf\limits_{t\in[0,T]}\lbrace \max(|X_t|,|Z_t|)\geq n\rbrace,\quad n\in\N.
  \end{equation*}
  Note that $\tau_n\to T$ a.s. as $n\to\infty$ since $X\in L^1(\Omega\times [0,T];\R^d)$ and $Z$ is continuous. Since the underlying filtered probability space satisfies the usual conditions, by the D{\'e}but theorem (see \cite[Chapter~I, (4.15) Theorem]{Revuz1999}), the hitting times $(\tau_n)_{n\in\N}$ are stopping times. Therefore, it remains to show that $(\tau_n)_{n\in \N}$ is a localizing sequence for $(\mathcal{M}_t^f)_{t\in[0,T]}$. For this purpose, we approximate $f$ by the functions $(f_n)_{n\in\N}\subset C^2_0(\R)$ given by $f_n := \phi_n f$ for some $\phi_n\in C_0^2(\R)$ taking values in $[0,1]$ and being identical to~$1$ on $[-n,n]$. Hence, $(\mathcal{M}_t^{f_n})_{t\in[0,T]}$ is a local martingale for every $n\in\N$ and, thus, the stopped process $(\mathcal{M}_{t\wedge \tau_n}^{f_n})_{t\in[0,T]}$, given by
  \begin{equation*}
    \mathcal{M}_{t\wedge\tau_n}^{f_n}=(f_n)(Z_{t\wedge\tau_n})-\int_0^{t\wedge\tau_n} \mathcal{A}^{f_n}(s,X_s,\L(X_s),Z_s)\dd s,\quad t\in[0,T],
  \end{equation*}
  is a martingale since
  \begin{equation*}
    |\mathcal{M}_{t\wedge\tau_n}^{f_n}|\leq \sup_{x\in \R}|f(x)| + C_{\sigma, b, n} n^2<\infty,
  \end{equation*}
  for some constant $C_{\sigma, b, n}>0$, using the definition of $\tau_n$ and the linear growth condition on $b$ and $\sigma$. Because $ \mathcal{M}_{t\wedge\tau_n}^{f_n}= \mathcal{M}_{t\wedge\tau_n}^{f}$ for $t\in [0,T]$, $(\mathcal{M}_{t\wedge\tau_n}^{f})_{t\in [0,T]}$ is a martingale for every $n\in \N$ and, thus, $(\tau_n)_{n\in \N}$ a localizing sequence for $(\mathcal{M}^{f}_t)_{t\in [0,T]}$.

  (ii) Since the stochastic process $(Z_t)_{t\in[0,T]}$ is a semimartingale with absolutely continuous characteristics
  \begin{equation*}
    \bigg(\int_0^{\cdot} b(s,X_s, \mathcal{L}(X_s))\dd s ,\int_0^{\cdot}(\sigma \sigma^\top)(s,X_s,\mathcal{L}(X_s))\dd s,0\bigg),
  \end{equation*}
  the assertion follows by \cite[Theorem~2.1.2]{Jacod2012}.
\end{proof}

Building on these preliminary observations and the martingale problem for stochastic Volterra equations (see \cite[Definition~2.4]{Promel2023_weak}), we formulate a local martingale problem associated to the distribution-dependent stochastic Volterra equation~\eqref{eq:MVSVE}. For a comparison to the classical martingale problem for ordinary stochastic differential equations we refer, e.g., to \cite[Definition~7.1.1]{Kallianpur2014}).

\begin{definition}\label{def: MP for stoch process}
  A \textup{solution to the distribution-dependent Volterra local martingale problem} given $(\mu_0,b,\sigma,K_b,K_\sigma)$ is a tuple $(Z,(\Omega,\mathcal{F},(\mathcal{F}_t)_{t\in [0,T]},\mathbb{P}))$ such that
  \begin{enumerate}
    \item[(i)] $(\Omega,\mathcal{F}, (\mathcal{F}_t)_{t\in [0,T]},\P)$ is a filtered probability space, where $(\mathcal{F}_t)_{t\in [0,T]}$ fulfills the usual conditions,
    \item[(ii)] $(Z_t)_{t\in[0,T]}$ is a continuous $d$-dimensional semimartingale with $Z_0 \sim \mu_0$,
    \item[(iii)]the process $(\mathcal{M}^f_t)_{t\in[0,T]}$, given by
    \begin{equation*}\label{def:M_f}
      \mathcal{M}^f_t:= f(Z_t)-\int_0^t \mathcal{A}^f(s,X_s,\L(X_s),Z_s)\dd s,\quad t\in [0,T],
    \end{equation*}
    is a local martingale for every $f\in C_0^2(\R^d)$, where $\mathcal{A}^f$ is defined as in \eqref{eq:operator A},
    \begin{equation}\label{eq:X_MP}
      X_t := Z_0 +\int_0^t K_b(s,t)\dd A_s + \int_0^tK_\sigma(s,t)\dd M_s, \quad t\in[0,T],
      \quad\P\text{-a.s.},
    \end{equation}
	$(A_t)_{t\in[0,T]}$ is predictable of bounded variation and $(M_t)_{t\in[0,T]}$ is a local martingale with $M_0=0$, such that $Z=A+M$.
  \end{enumerate}
\end{definition}

Classically, the concept of a (local) martingale problem associated to a stochastic differential equation is often introduced directly for probability measures on the canonical space, see, e.g. \cite[Chapter~6]{Stroock1979} and \cite[Chapter~5]{Karatzas1991}. To formulate an analogous martingale problem for distribution-dependent SVEs, we assume that
\begin{equation*}
  Z\colon C([0,T];\R^d) \to \R^d, \quad \text{with}
  \quad Z_t(\omega):= \omega(t)
  \quad \text{for }\omega\in C([0,T];\R^d),
\end{equation*}
is the canonical process on the path space $C([0,T];\R^d)$, where $C([0,T];\R^d)$ denotes the space of all continuous functions $\omega\colon [0,T]\to \R^d$ equipped with the supremum norm $\|\cdot\|_{\infty}$ and $\mathcal{B}(C([0,T];\R^d))$ the Borel $\sigma$-algebra on $C([0,T];\R^d)$.

\begin{definition}\label{def:martproblem_new}
  A probability measure $P$ on $(C([0,T];\R^d),\mathcal{B}(C([0,T];\R^d)))$ is called \textup{solution to the distribution-dependent Volterra local martingale problem} given $(\mu_0,b,\sigma,K_b,K_\sigma)$ if
  \begin{enumerate}
    \item[(i)] $(Z_t)_{t\in[0,T]}$ is a $d$-dimensional semimartingale and $Z_0 \sim \mu_0$, 
    \item[(ii)]the process $(\mathcal{M}^f_t)_{t\in[0,T]}$, given by
    \begin{equation}\label{def:M_f_new}
      \mathcal{M}^f_t:= f(Z_t)-\int_0^t \mathcal{A}^f(s,X_s,\L(X_s),Z_s)\dd s,\quad \mathcal{F}_t, \quad t\in [0,T],
    \end{equation}
    is a local martingale for every $f\in C_0^2(\R^d)$, where $\mathcal{A}^f$ is defined as in \eqref{eq:operator A},
    \begin{equation}\label{eq:X_MP_new}
      X_t := Z_0 +\int_0^t K_b(s,t)\dd A_s + \int_0^tK_\sigma(s,t)\dd M_s, \quad t\in[0,T],
      \quad P\text{-a.s},
    \end{equation}
	$(A_t)_{t\in[0,T]}$ is predictable of bounded variation and $(M_t)_{t\in[0,T]}$ is a local martingale with $M_0=0$, such that $Z=A+M$. Here $\mathcal{F}_t = \mathcal{G}_{t+}$, and $\{\mathcal{G}_t\}$ is the augmentation under $P$ of the canonical filtration $\mathcal{B}_t := \sigma(Z_s\, :\, s\in [0,t])$. % Für die Filtration, siehe KS S. 307 und S.314)
  \end{enumerate}
\end{definition}

\begin{remark}
  Condition (ii) of Definition~\ref{def: MP for stoch process} and Condition (i) of Definition~\ref{def:martproblem_new} can be relaxed to the condition ``$(Z_t)_{t\in [0,T]}$ is an $(\mathcal{F}_t)$-adapted and continuous process'', respectively. Indeed, this requirement together with the other conditions already implies the semimartingale property of $(Z_t)_{t\in [0,T]}$, see Lemma~\ref{lem:Zsemimart}. However, we choose to postulate the semimartingale property directly in order to ensure that the definition is obviously well-posed.
\end{remark}

\begin{remark}
  In the case of $K_{b}=K_{\sigma}=K$ condition \eqref{eq:X_MP_new} simplifies to
  \begin{equation*}
    X_t := Z_0 +\int_0^t K(s,t)\dd Z_s , \quad t\in[0,T], \quad\P\text{-a.s}.
  \end{equation*}
  Hence, it is straightforward to see that in the case of $K = 1$ we have $X = Z$ and the distribution-dependent Volterra local martingale problem reduces to the standard martingale problem for (McKean--Vlasov) stochastic differential equations, see e.g. \cite{Stroock1979,Funaki1984}. Moreover, if the coefficients $b$ and $\sigma$ are independent of the law $\mathcal{L}(X_{\cdot})$, the distribution-dependent Volterra local martingale problem in Definition~\ref{def: MP for stoch process} coincides with to the Volterra local martingale problem as stated in \cite[Definition~2.4]{Promel2023_weak}, and, assuming additionally that the kernels are of convolutional type, it is equivalent to the Volterra local martingale problem presented in \cite[Definition~3.1]{AbiJaber2021}.
\end{remark}

In the next proposition, we demonstrate that the existence of weak solutions to distribution-dependent SVEs is equivalent to the solvability of the associated Volterra local martingale problems.

\begin{proposition}\label{prop:equivalence}
  Suppose that $(\mu_0,b,\sigma,K_b,K_\sigma)$ satisfy Assumption~\ref{ass:short}. Then, the following statements are equivalent.
  \begin{enumerate}
    \item[(i)] There exists a weak solution $((X,B),(\Omega,\mathcal{F},\mathbb{P}),(\mathcal{F}_t)_{t\in [0,T]})$ to the distribution-dependent SVE~\eqref{eq:MVSVE} in the sense of Definition~\ref{def:weak solution}.

    \item[(ii)] There exists a solution $(Z,(\Omega,\mathcal{F},(\mathcal{F}_t)_{t\in [0,T]},\mathbb{P}))$ to the distribution-dependent Volterra local martingale problem in the sense of Definition~\ref{def: MP for stoch process}.

    \item[(iii)] There exists a solution $P$ to the distribution-dependent Volterra local martingale problem in the sense of Definition~\ref{def:martproblem_new}.
  \end{enumerate}
\end{proposition}

\begin{proof}
  Let $(X,B)$ be a (weak) solution to~\eqref{eq:MVSVE} on a probability space $(\Omega,\mathcal{F},\P)$ with a suitable filtration $(\mathcal{F}_t)_{t\in [0,T]}$ satisfying the usual conditions. Setting
  \begin{equation*}
    Z_t:=A_t+M_t:= X_0 + \int_0^t b(s,X_s,\mathcal{L}(X_s))\dd s+\int_0^t \sigma(s,X_s,\mathcal{L}(X_s))\dd B_s,\quad t\in [0,T],
  \end{equation*}
  (i) and \eqref{eq:X_MP} hold. It{\^o}'s formula applied to $f(Z_t)$ for $f\in C_0^2(\R^d)$ yields that
  \begin{align*}
    \mathcal{M}_t^f =f(Z_t)&-\int_0^t \sum_{i=1}^d b_i(s,X_s,\L(X_s))\frac{\partial f(Z_s)}{\partial z_i}\dd s \\
    &- \int_0^t \frac{1}{2}\sum_{i,j=1}^d (\sigma\sigma^\top)_{ij}(s,X_s,\L(X_s)) \frac{\partial^2 f(Z_s)}{\partial z_i \partial z_j}\dd s\\
    =f(Z_0)&+ \sum_{j=1}^m \int_0^t \sum_{i=1}^d \frac{\partial f(Z_s)}{\partial z_i} \sigma_{ij}(s,X_s,\mathcal{L}(X_s))\dd B^j_s,
  \end{align*}
  which is a local martingale and, by its definition, $Z$ is a semimartingale satisfying~\eqref{def:M_f}. Hence, $(Z,(\Omega,\mathcal{F},(\mathcal{F}_t)_{t\in [0,T]},\P))$ is a solution to the distribution-dependent Volterra local martingale problem in the sense of Definition~\ref{def: MP for stoch process}. Consequently, the probability measure $P := \mathbb{P} \circ Z^{-1}$, induced by $Z$ on $(C([0,T];\R^d),\mathcal{B}(C([0,T];\R^d)))$, is a solution to the distribution-dependent Volterra local martingale problem in the sense of Definition~\ref{def:martproblem_new}.

  Conversely, if there exists a solution $P$ to the distribution-dependent Volterra local martingale problem in the sense of Definition~\ref{def:martproblem_new}, we obtain a weak solution to the distribution-dependent SVE~\eqref{eq:MVSVE} by using \eqref{eq:X_MP_new} and Lemma~\ref{lem:Zsemimart}, which yields that
  \begin{equation*}
    A_t= X_0+\int_0^t b(s,X_s,\mathcal{L}(X_s))\dd s
    \quad\text{and}\quad
    M_t=\int_0^t \sigma(s,X_s,\mathcal{L}(X_s))\dd B_s
  \end{equation*}
  for a Brownian motion~$(B_t)_{t\in [0,T]}$ on a suitable filtered probability space $(\Omega,\mathcal{F},(\mathcal{F}_t)_{t\in [0,T]},\P)$.
  %
  %Definition~\ref{def:martproblem_new} can be considered a special case of Definition~\ref{def: MP for stoch process} with a solution of the former together with the filtered probability space defined therein directly matching the definition of the latter. Vice versa, when having a solution $(Z,(\Omega,\mathcal{F},(\mathcal{F}_t)_{t\in [0,T]},\mathbb{P}))$ to the distribution-dependent Volterra local martingale problem in Definition~\ref{def: MP for stoch process} we can take $P := \mathbb{P} \circ Z^{-1}$ to obtain a solution in the sense of Definition~\ref{def:martproblem_new}.
\end{proof}

\section{Existence of weak solutions to distribution-dependent SVEs}\label{sec:construction}

In this section we establish the existence of a solution to the distribution-dependent Volterra local martingale problem given $(\mu_0, b,\sigma,K_b,K_\sigma)$ which, by Proposition~\ref{prop:equivalence}, gives us the existence of a weak solution to the distribution-dependent SVE~\eqref{eq:MVSVE} with initial condition $\mathcal{L}(X_0)=\mu_0$. To that end, we need to impose suitable assumptions on the coefficients and kernels.

\begin{assumption}\label{ass:volatility}
  There is an $\eta \geq 1$ such that the coefficients $b: [0,T] \times \mathbb{R}^d \times \mathcal{P}_\eta(\mathbb{R}^d)\to \bR^{d}$ and $\sigma: [0,T] \times \mathbb{R}^d \times \mathcal{P}_\eta(\mathbb{R}^d)\to \bR^{d \times m}$ are measurable functions satisfying:
  \begin{itemize}
    \item there exists a constant $C > 0$ such that, for all $(t,x,\mu) \in [0,T] \times \mathbb{R}^d \times \mathcal{P}_\eta(\mathbb{R}^d)$,
    \begin{align*}
      |\sigma(t,x,\mu)| + |b(t,x,\mu)| &\leq C(1+|x|+W_\eta(\delta_0,\mu)),
    \end{align*}
    \item for every compact set $\mathcal K \subset \bR^d$ and every $\epsilon > 0$ there exists a $\delta > 0$ such that
    \begin{equation*}
      |b(t,x,\mu)-b(t,y,\nu)|+|\sigma(t,x,\mu)-\sigma(t,y,\nu)| \leq \epsilon
    \end{equation*}
    holds for all $t \in [0,T], x,y \in \mathcal K$ satisfying $|x-y| \leq \delta$, and $\mu, \nu \in \mathcal{P}_\eta(\bR^d)$ satisfying $W_\eta(\mu,\nu)\leq \delta$.
  \end{itemize}
\end{assumption}

\begin{assumption}\label{ass:kernel}
  There are constants $\gamma\in (0,\frac{1}{2}]$, $p\in [1,\infty)$ and $\epsilon>0$ such that:
  \begin{itemize}
    \item $K_b, K_\sigma\colon \Delta_T\to \mathbb{R}$ are measurable functions fulfilling, for some $L>0$,
    \begin{align*}
      &\int_0^t |K_{b}(s,t')-K_{b}(s,t)|^{1+\epsilon}\dd s + \int_t^{t'} |K_{b}(s,t')|^{1+\epsilon}\dd s \leq L|t'-t|^{\gamma(1+\epsilon)},\\
      &\int_0^t |K_{\sigma}(s,t')-K_{\sigma}(s,t)|^{2+\epsilon}\dd s  + \int_t^{t'} |K_{\sigma}(s,t')|^{2+\epsilon}\dd s  \leq L|t'-t|^{\gamma(2+\epsilon)},
    \end{align*}
    for all $(t,t^\prime)\in \Delta_T$;
    \item $\mu_0 \in \mathcal{P}_{p}(\R^d)$  for some $p > \max\{\frac{2 \eta + 1}{\gamma},\frac{4+2\varepsilon}{\varepsilon}\}$.
  \end{itemize}
\end{assumption}

%\begin{remark}
%  In the case of classical stochastic Volterra equation, i.e. no dependency on the law, the condition on the initial condition can be relaxed to $\mu_0 \in \mathcal{P}_{p}(\R^d)$ for some $p > \max\{\frac{1}{\gamma},\frac{4+2\varepsilon}{\varepsilon}\}$. Hence, Theorem~\ref{thm: weak solution to mean-field SVEs} is a generalization of the results in \cite{Promel2023_weak} even in the classical distribution-independent Volterra case.
%\end{remark}

In the following we fix the constants $p$,$\gamma$, and $\varepsilon$ as defined in Assumption~\ref{ass:kernel}.

\begin{remark}\label{rem:asskernel_to_assshort}
  Note that Assumption~\ref{ass:kernel} implies Assumption~\ref{ass:short} since
  \begin{equation*}
    \int_0^{t} |K_b(s,t)| \dd s \leq \int_0^{t} |K_b(s,t)|^{1+\epsilon} \dd s \leq L t^{\gamma(1+\epsilon)}
  \end{equation*}
  and
  \begin{equation*}
    \int_0^{t} |K_\sigma(s,t)|^2 \dd s \leq\int_0^{t} |K_\sigma(s,t)|^{2+\epsilon} \dd s \leq L t^{\gamma(2+\epsilon)},
  \end{equation*}
  that is, $K_b(\cdot,t) \in L^1([0,t])$ and $K_\sigma(\cdot,t) \in L^2([0,t])$, for every $ t \in [0,T]$. Moreover, let us remark, because of
  \begin{align*}
    \int_0^T \int_0^{t} |K_b(s,t)| \dd s \dd t \leq L \int_0^T t^{\gamma(1+\epsilon)} \dd t \leq \frac{LT^{\gamma(1+\epsilon)+1}}{\gamma (1 + \epsilon)+1}
  \end{align*}
  and
  \begin{align*}
    \int_0^T \int_0^{t} |K_\sigma(s,t)|^2 \dd s \dd t \leq L \int_0^T t^{\gamma(2+\epsilon)} \dd t \leq \frac{LT^{\gamma(2+\epsilon)+1}}{\gamma (2 + \epsilon)+1},
  \end{align*}
  we have $K_b \in L(\Delta_T)$ and $K_\sigma \in L^2(\Delta_T)$.
\end{remark}

\begin{remark} 
  Assumption~\ref{ass:kernel} is satisfied, e.g., by the following type of diffusion kernels:
  \begin{enumerate}
    \item[(i)] $K_\sigma(s,t):=C(t-s)^{-\alpha}$ for $\alpha\in(0,\frac{1}{2})$ with $\varepsilon \in (0,\frac{1}{\alpha}-2)$ and $\gamma \in (0,\frac{1}{2+\varepsilon}-\alpha]$,
    \item[(ii)] $K_\sigma(s,t):=\tilde{K}(t-s)$ for a Lipschitz continuous function $\tilde{K}\colon [0,T]\to \R$,
    \item[(iii)] weakly differentiable kernels such that $\partial_1 K_\sigma(s,t)\leq C(t-s)^{-\alpha}$ for $\alpha\in(0,\frac{1}{2})$, and
    \item[(iv)] kernels fulfilling \cite[Assumption~2.1]{Promel2023},
    \item[(v)] a finite combination of exponential kernel $K_\sigma(s,t) = \sum_{i=1}^k c_i \exp(- \theta_i (t-s))$ with $k \in \bN$, $c_i >0$ and $\theta_i \in \bR^+$ for $i = 1, \ldots, k$ with $\gamma \in (0, \frac{1}{2+ \varepsilon}]$,
      \item[(vi)] the gamma kernel $K_\sigma(s,t) = \frac{1}{\Gamma(\alpha)} \exp(- \beta (t-s)) (t-s)^{\alpha -1}$ with $\beta >0$ and exponents $\alpha \in (\frac{1}{2},1)$ with $\varepsilon< \frac{2 \alpha - 1}{1-\alpha}$ and $\gamma = \alpha - \frac{1+\varepsilon}{2+\varepsilon}$.
  \end{enumerate}
  In all cases one needs an initial condition with $p> \max\{\frac{2 \eta +1}{\gamma}, \frac{4+2 \varepsilon}{\varepsilon}\}$ finite moments.
\end{remark}

The next theorem presents the existence of weak solutions to the distribution-dependent SVEs.

\begin{theorem}\label{thm: weak solution to mean-field SVEs}
  Suppose Assumption~\ref{ass:volatility} and~\ref{ass:kernel}. Then, there exists a weak solution $(X_t)_{t\in [0,T]}$ to the distribution-dependent stochastic Volterra equation~\eqref{eq:MVSVE}.
\end{theorem}

\begin{proof}
  Theorem~\ref{thm: weak solution to mean-field SVEs} is an immediate consequence of Proposition~\ref{prop:equivalence} and Lemma~\ref{lem:solution to Volterra martingale problem} below, which establishes the existence of a solution to distribution-dependent Volterra local martingale problem given $(\mu_0, b,\sigma,K_b,K_\sigma)$. Note that Proposition~\ref{prop:equivalence} is applicable by Remark~\ref{rem:asskernel_to_assshort}.
\end{proof}

The remainder of this section is devoted to establish the existence of a solution to the Volterra martingale problem given $(\mu_0, b,\sigma,K_b,K_\sigma)$, see Lemma~\ref{lem:solution to Volterra martingale problem} below, and, consequently, the existence of a weak solution to distribution-dependent SVE~\eqref{eq:MVSVE}. For this purpose, we consider an Euler-type approximation of the distribution-dependent SVE~\eqref{eq:MVSVE}.

\medskip

Let $(\pi_m = \{t_0^m=0, \ldots, t_{N(m)}^m=T\})_{m \in \bN}$ be a sequence of partitions with $\lim_{m \to \infty} |\pi_m|=0$, where $|\pi_m| := \sup_{i\leq N(m)} t^m_i - t^m_{i-1}$ for every $m \in \bN$ is the mesh size. We introduce, for every $m \in \bN$, the function $\kappa_m \colon [0,T] \to [0,T]$ by
\begin{equation*}
  \kappa_m(T) := T
  \quad \text{and} \quad
  \kappa_m(t) := t_i^m \text{ for } t_i^m \leq t < t_{i+1}^m, ~i=0,1,\ldots,N(m)-1.
\end{equation*}
For every $m \in \bN$, we iteratively define the process $(X^m(t))_{t \in [0,T]}$ by $X^m(0):=X_0$ and for $t \in (t_i^m,t_{i+1}^m]$, $i=0,1,\ldots,N(m)-1$, by
\begin{align}\label{eq:Euler_approx_process}
  \begin{split}
  X^m(t) &:= X_0 + \int_0^{t_i^m} K_b(s,t) b\big(s, X^m(\kappa_m(s)), \L(X^m(\kappa_m(s)))\big) \dd s \\
  & \qquad + \int_{t_i^m}^{t} K_b(s,t) b\big(s, X^m(t_i^m), \L(X^m(t_i^m))\big) \dd s\\
  & \qquad + \int_0^{t_i^m} K_\sigma(s,t) \sigma\big(s, X^m(\kappa_m(s)), \L(X^m(\kappa_m(s)))\big) \dd B_s \\
  & \qquad + \int_{t_i^m}^{t} K_\sigma(s,t) \sigma\big(s, X^m(t_i^m), \L(X^m(t_i^m))\big) \dd B_s.
  \end{split}
\end{align}

Moreover, we use the notation $A_{\xi}\lesssim B_{\xi}$ for a generic parameter~$\xi$, meaning that $A_{\xi}\le CB_{\xi}$ for some constant $C>0$ independent of~$\eta$. With $\mathcal{H}^2$ we denote the space of square integrable continuous martingales $(M_t)_{t\in [0,T]}$ equipped with the norm $\|M\|_{\mathcal{H}^2} := \bE[\sup_{t\in [0,T]} M_t^2]^{1/2}$ and we write $\xrightarrow{\bP}$ for convergence in probability.

\begin{remark}\label{rmk:BDG inequality}
  In general, the stochastic process
  \begin{align*}
    \Big( \int_0^t K_\sigma(s,t)\sigma(s,X_s) \dd B_s\Big)_{t \in [0,T]}
  \end{align*}
  is not a local martingale, assuming sufficient regularity to ensure its well-posedness like Assumption~\ref{ass:kernel} and~\ref{ass:volatility}. Yet, fixing the second argument in the kernel, the stochastic process
  \begin{align*}
    \Big( \int_0^t K_\sigma(s,\tilde{t})\sigma(s,X_s) \dd B_s\Big)_{t \in [0,\tilde{t}]}
  \end{align*}
  is a local martingale for every $\tilde{t} \in (0,T]$. Hence, we can apply the classical Burkholder--Davis--Gundy inequality to this martingale to obtain
  \begin{align*}
    \bE \Big[ \sup_{t \in [0,\tilde{t}]} \Big| \int_0^t K_\sigma(s,\tilde{t}) \sigma(s,X_s)\dd B_s\Big|^p \Big] \leq C_p \bE\Big[ \Big( \int_0^{\tilde{t}} (K_\sigma(s,\tilde{t}) \sigma (s,X_s))^2 \dd s \Big)^{p/2} \Big], \quad \tilde{t} \in [0,T],
  \end{align*}
  for some positive $C_p$ and, in particular,
  \begin{align}\label{eq:BDG}
    \bE \Big[ \Big| \int_0^{\tilde{t}} K_\sigma(s,\tilde{t}) \sigma(s,X_s)\dd B_s \Big|^p \Big] \leq C_p \bE\Big[ \Big( \int_0^{\tilde{t}} (K_\sigma(s,\tilde{t}) \sigma (s,X_s))^2 \dd s \Big)^{p/2} \Big]
  \end{align}
for every $\tilde{t} \in [0,T]$. In the following, we often refer to the Burkholder--Davis--Gundy inequality meaning that we use \eqref{eq:BDG}.
\end{remark}

As first step, we derive that the Euler-type approximation, defined in \eqref{eq:Euler_approx_process}, is (uniformly) bounded in $L^p$-spaces.

\begin{lemma}\label{lemma:boundedness}
  Suppose Assumption~\ref{ass:volatility} and~\ref{ass:kernel}. Then, for every $q \in [1,p]$ there is a $C_q>0$ such that $\bE[|X^m(t)|^q] \leq C_q$ holds for all $m \in \bN$ and $t \in [0,T]$. In particular, $X^m \in L^q(\Omega \times [0,T])$ holds for every $m \in \bN$ and $q \in [1,p]$.
\end{lemma}

\begin{proof}
  Fix some $m \in \bN$ and define $\tilde{p}:=\frac{p}{p-2}\leq 1+ \frac{\epsilon}{2}$, $p':= \frac{p}{p-1}\leq 1+ \epsilon$. For $t \in (t_i^m, t_{i+1}^m]$, $i=1, \ldots, N(m)-1$, we iteratively get
  \begin{align*}
    \bE[|X^m(t)|^p] & = \bE\Big[\Big|X_0 + \int_0^t K_b(s,t) b \big(s,X^m(\kappa_m(s)),\L(X^m(\kappa_m(s)))\big) \dd s\\
    & \qquad + \int_0^t K_\sigma(s,t) \sigma\big(s,X^m(\kappa_m(s)),\L(X^m(\kappa_m(s)))\big) \dd B_s\Big|^p\Big]\\
    & \leq C_p \Big(\bE[|X_0|^p] + \bE\Big[\Big|\int_0^t K_b(s,t) b\big(s,X^m(\kappa_m(s)),\L(X^m(\kappa_m(s)))\big) \dd s\Big|^p\Big]\\
    & \qquad + \bE\Big[\Big|\int_0^t K_\sigma(s,t) \sigma\big(s,X^m(\kappa_m(s)),\L(X^m(\kappa_m(s)))\big) \dd B_s\Big|^p\Big]\Big)\\
    & \leq C_p \Big(\bE[|X_0|^p] + \bE\Big[\Big|\int_0^t K_b(s,t) b\big(s,X^m(\kappa_m(s)),\L(X^m(\kappa_m(s)))\big) \dd s\Big|^p\Big]\\
    & \qquad + \bE\Big[\Big(\int_0^t |K_\sigma(s,t) \sigma\big(s,X^m(\kappa_m(s)),\L(X^m(\kappa_m(s)))\big)|^2 \dd s\Big)^{p/2}\Big]\Big)\\
    & \leq C_{p,T} \Big(\bE[|X_0|^p] \\
    & \qquad +\Big(\int_0^t |K_b(s,t)|^{p'} \dd s\Big)^\frac{p}{p'} \int_0^t \bE[|b\big(s,X^m(\kappa_m(s)),\L(X^m(\kappa_m(s)))\big)|^p]\dd s\\
    & \qquad + \Big(\int_0^t |K_\sigma(s,t)|^{2 \tilde{p}} \dd s\Big)^\frac{p}{2 \tilde{p}} \int_0^t \bE[|\sigma\big(s,X^m(\kappa_m(s)),\L(X^m(\kappa_m(s)))\big)|^p]\dd s\Big),
  \end{align*}
  where we used the Burkholder--Davis--Gundy inequality (cf. Remark~\ref{rmk:BDG inequality}) in the penultimate inequality and the H{\"o}lder inequality in the last step.

  Recall that $p \geq \eta$ and note that
  \begin{equation*}
    W_\eta(\delta_0,\L(X^m(\kappa_m(s))))^p \leq \bE[|X^m(\kappa_m(s))|^\eta]^{\frac{p}{\eta}} \leq \bE[|X^m(\kappa_m(s))|^p].
  \end{equation*}
  Using the linear growth assumption in Assumption~\ref{ass:volatility}, the above fact and Assumption~\ref{ass:kernel}, we have
  \begin{align}\label{eq:boundedness_1}
    \bE[|X^m(t)|^p] & \leq C_{p,T} \Big(\bE[|X_0|^p] + C_{K,b,\sigma} \int_0^t 1 + 2 \bE[|X^m(\kappa_m(s))|^p] \dd s \Big).
  \end{align}
  Then, we have
  \begin{align*}
    \bE[|X^m(t_i^m)|^p] & \leq C_{p,T} \Big(\bE[|X_0|^p] + C_{K,b,\sigma} \int_0^{t_i^m} 1 + 2 \bE[|X^m(\kappa_m(s))|^p] \dd s \Big)\\
    & = C_{p,T} \bE[|X_0|^p] + C_{p,T,K,b,\sigma} \sum_{j=0}^{i-1} (t_{j+1}^m - t_j^m) ( 1 + 2 \bE[|X^m(t_j^m)|^p])\\
    & \leq C_{p,T} \bE[|X_0|^p] + C_{p,T,K,b,\sigma} T + \sum_{j=0}^{i-1} 2C_{p,T,K,b,\sigma} (t_{j+1}^m - t_j^m) \bE[|X^m(t_j^m)|^p].
  \end{align*}
  Applying the discrete Gronwall lemma (see, e.g., \cite[Lemma 1.4.2]{Quarteroni1994}), we get
  \begin{align}\label{eq:boundedness_2}
    \begin{split}
    \bE[|X^m(t_i^m)|^p] & \leq (C_{p,T} \bE[|X_0|^p] + C_{p,T,K,b,\sigma} T) \exp\Big(\sum_{j=0}^{i-1} 2C_{p,T,K,b,\sigma} (t_{j+1}^m - t_j^m)\Big)\\
    & = (C_{p,T} \bE[|X_0|^p] + C_{p,T,K,b,\sigma} T) \exp( 2C_{p,T,K,b,\sigma} t_i^m).
    \end{split}
  \end{align}
  Then, for $t \in [0,T]$, by inserting \eqref{eq:boundedness_2} into \eqref{eq:boundedness_1}, we obtain
  \begin{align*}
    &\bE[|X^m(t)|^p] \\
    &\quad \leq C_{p,T} \Big(\bE[|X_0|^p] + C_{K,b,\sigma} \int_0^t 1 + 2 (C_{p,T} \bE[|X_0|^p] + C_{p,T,K,b,\sigma} T) \exp( 2C_{p,T,K,b,\sigma} \kappa_m(s)) \dd s\Big)\\
    &\quad \leq C_{p,T} \Big(\bE[|X_0|^p] + C_{K,b,\sigma}t+C_{K,b,\sigma} \int_0^t  2 (C_{p,T} \bE[|X_0|^p] + C_{p,T,K,b,\sigma} T) \exp( 2C_{p,T,K,b,\sigma} s) \dd s \Big)\\
    &\quad = C_{p,T} \Big(\bE[|X_0|^p] + C_{K,b,\sigma}t+C_{K,b,\sigma} \Big(C_{p,T,K,b,\sigma} \bE[|X_0|^p] +  T\Big) (\exp( 2C_{p,T,K,b,\sigma} t)-1)\Big)\\
    &\quad=: g_p(t),
  \end{align*}
  where $g_p$ is a continuous function independent of $m$. Using Fubini theorem, we observe that
  \begin{equation*}
    \bE\Big[\int_0^T |X^m(t)|^p \dd t\Big] \leq \int_0^T g_p(t)\dd t \leq T \sup_{t \in [0,T]} g_p(t).
  \end{equation*}
  By the orderedness of $L^p$-spaces the general statement follows.
\end{proof}

As second step, we establish the tightness of the Euler-type approximation~\eqref{eq:Euler_approx_process} and, thus, the existence of a converging subsequence. To that end, for $t\in [0,T]$, the marginal distribution~$\mu_t$ at time $t$ of the probability measure~$\mu$ on $C([0,T];\R^d)$ is defined as the pushforward measure, i.e.
\begin{equation*}
  \mu_t(A) := \mu(\{x \in C([0,T]; \bR^d) \colon x_t \in A\}), \quad t\in [0,T],
\end{equation*}
for any measurable set $A \subset \bR^d$.

\begin{remark}\label{remark_Wasserstein}
  Note that, for all $\mu, \nu \in \mathcal{P}_\eta(C([0,T]; \bR^d))$ and $t \in [0,T]$,
  \begin{align*}
    W_\eta(\mu_t,\nu_t)^\eta & = \inf_{\pi \in \Pi(\mu_t,\nu_t)} \int_{(\bR^d)^{\otimes 2}} |x-y|^\eta \dd \pi(x,y)\\
    & = \inf_{\pi \in \Pi(\mu,\nu)} \int_{(\bR^d)^{\otimes 2}} |x-y|^\eta \dd \pi_t(x,y)\\
    & = \inf_{\pi \in \Pi(\mu,\nu)} \int_{C([0,T];\bR^d)^{\otimes 2}} |x_t-y_t|^\eta \dd \pi(x,y)\\
    & \leq \inf_{\pi \in \Pi(\mu,\nu)}  \int_{C([0,T];\bR^d)^{\otimes 2}} \sup_{t \in [0,T]}|x_t-y_t|^\eta \dd \pi(x,y)\\
    & = W_\eta(\mu,\nu)^\eta.
  \end{align*}
\end{remark}

We introduce the sequences $(A^m)_{m\in\N}$ and $(M^m)_{m\in\N}$ by 
\begin{equation}\label{eq:Am}
  A^m_t:= X^m(0)+\int_0^t b\big(s,X^m(\kappa_m(s)), \L(X^m(\kappa_m(s)))\big)\dd s,
\end{equation}
and
\begin{equation}\label{eq:Mm}
  M^m_t:=\int_0^t \sigma\big(s,X^m(\kappa_m(s)), \L(X^m(\kappa_m(s)))\big)\dd B_s, \qquad t\in[0,T].
\end{equation}
In the following, we denote $X\stackrel{\mathscr{D}}{\sim}Y$ for equality in law of stochastic processes $X$ and $Y$.

\begin{lemma}\label{lem:limitprocesses}
  Suppose Assumption~\ref{ass:volatility} and~\ref{ass:kernel}. Let $(A^m)_{m \in \bN}$, $(M^m)_{m \in \bN}$ and $(X^m)_{m \in \bN}$ be as in \eqref{eq:Am}, \eqref{eq:Mm} and \eqref{eq:Euler_approx_process}, respectively. Then, there is a subsequence $(X^{m_k},A^{m_k},M^{m_k})_{k \in \bN}$ of $(X^m,A^m,M^m)_{m \in \bN}$ and a probability space $(\tilde{\Omega}, \tilde{\mathcal{F}},\tilde{\bP})$ with stochastic processes $\hat{X}^k,\hat{A}^k,\hat{M}^k$, $k \in \bN$, $\tilde{X}, \tilde{A}, \tilde{M}$ and a Brownian motion $(\tilde{B}_t)_{t\in [0,T]}$ on it with $(\hat{X}^k, \hat{A}^k, \hat{M}^k) \stackrel{\mathscr{D}}{\sim} (X^{m_k},A^{m_k},M^{m_k})$, $k \in \bN$, such that $(\hat{X}^k, \hat{A}^k, \hat{M}^k) \to (\tilde{X},\tilde{A},\tilde{M})$ in $C([0,T]; \bR^d \times \bR^d \times \bR^d)$ as $k \to \infty$ $\tilde{\bP}$-a.s., where $\tilde{M}$ is a local martingale with representation
  \begin{equation*}
    \tilde{M}_t=\int_0^t \sigma(s,\tilde{X}_s, \L(\tilde{X}_s)) \dd \tilde{B}_s, \qquad t \in [0,T].
  \end{equation*}
\end{lemma}

\begin{proof}
  We aim to apply the Kolmogorov tightness criterion (see \cite[Problem~2.4.11]{Karatzas1991}). By assumption there is a $C >0$ such that $\bE[|A^m(0)|^p]=\bE[|X^m(0)|^p] \leq C$ holds for all $m \in \bN$. Furthermore, $M^m(0)=0$ for all $m \in \bN$. Recall $p > \max\{\frac{1+2\eta}{\gamma}, 2 + \frac{4}{\epsilon}\}$. The laws of the stochastic processes are uniformly bounded by Lemma~\ref{lemma:boundedness}. Among others, by the linear growth condition this implies that there are $C_b$ and $C_\sigma$ such that
  \begin{align*}
    |\sigma(s,x,\L(X^m(t)))| &\leq C_\sigma(1+|x|)
    \quad \text{and}\quad
    |b(s,x,\L(X^m(t)))| \leq C_b(1+|x|)
  \end{align*}
  for all $m \in \bN$, $s,t \in [0,T]$, and $x \in \bR$, since
  \begin{equation*}
    |\sigma(s,x,\L(X^m(t)))| \leq C\big(1+|x|+W_\eta(\delta_0,\L(X^m(t)))\big) \leq C(1+|x|+\bE[|X_m(t)|^\eta]^{1/\eta})
  \end{equation*}
  and an analogous inequality holds for $b$.

  For $(t,t') \in \Delta_T$ we have
  \begin{align*}
    |X^m(t') - X^m(t)|^p & = \Big|\int_0^{t'} K_b(s,t') b\big(s, X^m(\kappa_m(s)), \L(X^m(\kappa_m(s)))\big) \dd s\\
    & \qquad -\int_0^t K_b(s,t) b\big(s, X^m(\kappa_m(s)), \L(X^m(\kappa_m(s)))\big) \dd s\\
    & \qquad + \int_0^{t'} K_\sigma(s,t') \sigma\big(s, X^m(\kappa_m(s)), \L(X^m(\kappa_m(s)))\big) \dd B_s\\
    & \qquad -\int_0^t K_\sigma(s,t) \sigma\big(s, X^m(\kappa_m(s)), \L(X^m(\kappa_m(s)))\big) \dd B_s \Big|^p\\
    &\leq C_p \Big(\Big|\int_0^{t} (K_b(s,t')-K_b(s,t)) b\big(s, X^m(\kappa_m(s)), \L(X^m(\kappa_m(s)))\big) \dd s\Big|^p \\
    & \qquad + \Big|\int_t^{t'} K_b(s,t') b\big(s, X^m(\kappa_m(s)), \L(X^m(\kappa_m(s)))\big) \dd s \Big|^p\\
    & \qquad + \Big|\int_0^{t} (K_\sigma(s,t')-K_\sigma(s,t)) \sigma\big(s, X^m(\kappa_m(s)), \L(X^m(\kappa_m(s)))\big) \dd B_s\Big|^p \\
    & \qquad + \Big|\int_t^{t'} K_\sigma(s,t') \sigma\big(s, X^m(\kappa_m(s)), \L(X^m(\kappa_m(s)))\big) \dd B_s \Big|^p\Big)\\
    & =: \mathbf{A}+\mathbf{B}+\mathbf{C}+\mathbf{D}.
  \end{align*}

  For $\mathbf{A}$, we obtain
  \begin{align*}
    \bE[\mathbf{A}] & \leq C_p \Big( \int_0^t |K_b(s,t')-K_b(s,t)|^{1+\epsilon} \dd s\Big)^{\frac{p}{1+\epsilon}}\\
    & \qquad \times \bE\Big[ \Big| \int_0^t |b\big(s, X^m(\kappa_m(s)), \L(X^m(\kappa_m(s)))\big)|^{\frac{1+\epsilon}{\epsilon}} \dd s \Big|^{\frac{p \epsilon}{1+\epsilon}}\Big]\\
    & \leq C_p (L(t'-t)^{\gamma(1+\epsilon)})^{\frac{p}{1+ \epsilon}} \bE \Big[ \Big| \int_0^t (C_b(1+|X^m(\kappa_m(s))|)^{\frac{1+\epsilon}{\epsilon}} \dd s \Big|^{\frac{p \epsilon}{1+\epsilon}}\Big]\\
    & \leq C_{p,L,\epsilon} (t'-t)^{\gamma p} \bE \Big[C_{b,p} \Big| \int_0^t (1+ |X^m(\kappa_m(s))|)^{\frac{1+\epsilon}{\epsilon}}  \dd s \Big|^{\frac{p \epsilon}{1+\epsilon}} \Big]\\
    & \leq C_{p,L,\epsilon} (t'-t)^{\gamma p} \bE \Big[C_{b,p,\epsilon,T} + C_{b,p,\epsilon} \Big| \int_0^t |X^m(\kappa_m(s))|^{\frac{1+\epsilon}{\epsilon}}  \dd s \Big|^{\frac{p \epsilon}{1+\epsilon}} \Big]\\
    & \leq C_{p,L,\epsilon} (t'-t)^{\gamma p} \Big(C_{b,p,\epsilon,T} + C_{b,p,\epsilon} \bE \Big[t^{\frac{p \epsilon}{1+\epsilon}-1} \int_0^t |X^m(\kappa_m(s))|^p  \dd s \Big]\Big)\\
    & \leq C_{p,L,\epsilon} (t'-t)^{\gamma p} \Big(C_{b,p,\epsilon,T} + C_{b,p,\epsilon} t^{\frac{p \epsilon}{1+\epsilon}-1} \int_0^t \sup_{r \in [0,T]} \bE [|X^m(\kappa_m(r))|^p ]  \dd s \Big)\\
    & \leq C_{p,L,\epsilon,T,b} (t'-t)^{\gamma p},
  \end{align*}
  where we used the linear growth as well as the kernel assumptions as well as $\frac{p \epsilon}{1+\epsilon}>1$ for the Jensen inequality and Lemma~\ref{lemma:boundedness} in the last step.

  For $\mathbf{B}$, we see that
  \begin{align*}
    \bE[\mathbf{B}] &= C_p \bE\Big[ \Big| \int_t^{t'}K(s,t') b\big(s,X^m(\kappa_m(s)),\L(X^m(\kappa_m(s)))\big) \dd s \Big|^p \Big]\\
    & \leq C_p \Big( \int_t^{t'} |K(s,t')|^{1+\epsilon} \dd s \Big)^\frac{p}{1+\epsilon} \bE \Big[ \Big| \int_t^{t'}|b\big(s,X^m(\kappa_m(s)),\L(X^m(\kappa_m(s)))\big)|^\frac{1+\epsilon}{\epsilon} \dd s \Big|^{\frac{p \epsilon}{1+\epsilon}}\Big]\\
    & \leq C_p (L(t'-t)^{\gamma(1+\epsilon)})^\frac{p}{1+ \epsilon} \bE\Big[ \Big| \int_0^{t'}(C_b(1+|X^m(\kappa_m(s))|))^{\frac{1+\epsilon}{\epsilon}} \dd s \Big|^{\frac{p \epsilon}{1+\epsilon}}\Big]
  \end{align*}
  and we end up in the same situation as above.

  For $\mathbf{C}$, using the Burkholder--Davis--Gundy inequality and H{\"o}lder inequality, we get
  \begin{align*}
    \bE[\mathbf{C}] & \leq C_p \bE \Big[\Big|\int_0^{t} (K_\sigma(s,t')-K_\sigma(s,t))^2 \sigma(s, X^m(\kappa_m(s)), \L(X^m(\kappa_m(s))))^2 \dd s\Big|^{\frac{p}{2}} \Big]\\
    & \leq C_p \Big( \int_0^t |K_\sigma(s,t')-K_\sigma(s,t)|^{2+\epsilon} \dd s \Big)^{\frac{p}{2+ \epsilon}} \\
    & \qquad \times\bE \Big[ \Big| \int_0^t |\sigma(s, X^m(\kappa_m(s)), \L(X^m(\kappa_m(s))))|^{\frac{4+2 \epsilon}{\epsilon}} \dd s \Big|^{\frac{p \epsilon}{4 + 2 \epsilon}} \Big].
  \end{align*}
  Using the linear growth condition and the kernel assumption, we observe that
  \begin{align*}
    \bE[\mathbf{C}] & \leq C_p (L(t'-t)^{\gamma(2+\epsilon)})^{\frac{p}{2+ \epsilon}} \bE \Big[ \Big| \int_0^t (C_{\sigma}(1+ |X^m(\kappa_m(s))|))^{\frac{4+2 \epsilon}{\epsilon}} \dd s \Big|^{\frac{p \epsilon}{4+2 \epsilon}} \Big]\\
    & \leq C_{p,L,\epsilon} (t'-t)^{\gamma p} \bE \Big[ \Big| \int_0^t C_{\sigma,\epsilon}(1+ |X^m(\kappa_m(s))|^{\frac{4+2 \epsilon}{\epsilon}}) \dd s \Big|^{\frac{p \epsilon}{4+2 \epsilon}} \Big]\\
    & \leq C_{p,L,\epsilon} (t'-t)^{\gamma p} \bE \Big[ \Big|  C_{\sigma,\epsilon}t+ C_{\sigma,\epsilon}\int_0^t|X^m(\kappa_m(s))|^{\frac{4+2 \epsilon}{\epsilon}} \dd s \Big|^{\frac{p \epsilon}{4+2 \epsilon}} \Big]\\
    & \leq C_{p,L,\epsilon} (t'-t)^{\gamma p} \Big( C_{\sigma,\epsilon,p,T}+ C_{\sigma,\epsilon,p} \bE \Big[ \Big|\int_0^t|X^m(\kappa_m(s))|^{\frac{4+2 \epsilon}{\epsilon}} \dd s \Big|^{\frac{p \epsilon}{4+2 \epsilon}} \Big] \Big)\\
    & \leq C_{p,L,\epsilon} (t'-t)^{\gamma p} \Big(   C_{\sigma,\epsilon,p,T}+ C_{\sigma,\epsilon,p} \bE \Big[ t^{\frac{p \epsilon}{4+2 \epsilon}-1} \int_0^t |X^m(\kappa_m(s))|^p \dd s \Big] \Big)\\
    & = C_{p,L,\epsilon} (t'-t)^{\gamma p} \Big(   C_{\sigma,\epsilon,p,T}+ C_{\sigma,\epsilon,p}  t^{\frac{p \epsilon}{4+2 \epsilon}-1} \int_0^t\bE[|X^m(\kappa_m(s))|^p]\dd s \Big)\\
    & \leq C_{p,L,\epsilon} (t'-t)^{\gamma p} \Big(   C_{\sigma,\epsilon,p,T}+ C_{\sigma,\epsilon,p}  t^{\frac{p \epsilon}{4+2 \epsilon}-1} \int_0^t\sup_{r \in [0,T]}\bE[|X^m(\kappa_m(r))|^p]\dd s \Big)\\
    & \leq C_{p,T,\epsilon,L,\sigma} (t'-t)^{\gamma p},
  \end{align*}
  where we used $\frac{p \epsilon}{4+2 \epsilon} > 1$ for the Jensen inequality in the fourth last step and Lemma~\ref{lemma:boundedness} in the last step.

  For $\mathbf{D}$, using the Burkholder--Davis--Gundy inequality and H{\"o}lder inequality, we see that
  \begin{align*}
    \bE[\mathbf{D}] & \leq C_p \bE \Big[\Big|\int_t^{t'} K_\sigma(s,t')^2 \sigma\big(s, X^m(\kappa_m(s)), \L(X^m(\kappa_m(s)))\big)^2 \dd s\Big|^{\frac{p}{2}} \Big]\\
    & \leq C_p \Big( \int_t^{t'} |K_\sigma(s,t')|^{2+\epsilon} \dd s \Big)^{\frac{p}{2+ \epsilon}} \bE \Big[ \Big| \int_t^{t'} |\sigma\big(s, X^m(\kappa_m(s)), \L(X^m(\kappa_m(s)))\big)|^{\frac{4+2 \epsilon}{\epsilon}} \dd s \Big|^{\frac{p \epsilon}{4 + 2 \epsilon}} \Big].
  \end{align*}

  Using the linear growth condition and the kernel assumption we get
  \begin{align*}
    \bE[\mathbf{D}] & \leq C_p (L(t'-t)^{\gamma(2+\epsilon)})^{\frac{p}{2+ \epsilon}} \bE \Big[ \Big| \int_t^{t'} (C_{\sigma}(1+ |X^m(\kappa_m(s))|))^{\frac{4+2 \epsilon}{\epsilon}} \dd s \Big|^{\frac{p \epsilon}{4+2 \epsilon}} \Big]\\
    & \leq C_p (L(t'-t)^{\gamma(2+\epsilon)})^{\frac{p}{2+ \epsilon}} \bE \Big[ \Big| \int_0^{t'} (C_{\sigma}(1+ |X^m(\kappa_m(s))|))^{\frac{4+2 \epsilon}{\epsilon}} \dd s \Big|^{\frac{p \epsilon}{4+2 \epsilon}} \Big],
  \end{align*}
  and again we end up in the same situation as above.

  Combining these estimates we have
  \begin{equation}\label{eq:tightness}
    \bE[|X^m(t') - X^m(t)|^p] \leq C (t'-t)^{\gamma p}
  \end{equation}
  for $t',t \in [0,T]$ with $\gamma p > 1$.

  For $A^m$, we obtain
  \begin{align*}
    \bE[|A^m(t')-A^m(t)|^p] & = \bE\Big[ \Big| \int_t^{t'} b\big(s,X^m(\kappa_m(s)),\L(X^m(\kappa_m(s)))\big) \dd s \Big|^p \Big]\\
    & \leq C_{p,b} \bE\Big[ \Big| \int_t^{t'} 1 + |X^m(\kappa_m(s))| \dd s |^p \Big]\\
    & \leq C_{p,b} \bE\Big[ (t'-t)^{p-1} \int_t^{t'} |1+X^m(\kappa_m(s))|^p \dd s \Big]\\
    & \lesssim C_{p,b} (t'-t)^{p-1} \int_t^{t'} 1+\bE[|X^m(\kappa_m(s))|^p \dd s]\\
    & \leq C_{p,b} (t'-t)^{p-1} \Big((t'-t) +  \int_t^{t'} \sup_{r \in [0,T]} \bE[|X^m(r)|^p] \dd s \Big)\\
    & = C_{p,b} (t'-t)^{p-1} ((t'-t)+C(t'-t))\\
    & \leq C_{b,p,T} (t'-t)^{\gamma p}.
  \end{align*}
  For $M^m$, using the Burkholder--Davis--Gundy inequality, the linear growth condition and the boundedness of $\bE[|X_m(s)|^p]$, we see
  \begin{align*}
    \bE[|M^m(t')-M^m(t)|^p]	& = \bE\Big[\Big|\int_t^{t'} \sigma(s, X^m(\kappa_m(s)),\L(X^m(\kappa_m(s)))) \dd B_s\Big|^p\Big]\\
    & \leq C_p \bE\Big[\Big|\int_t^{t'} \sigma^2(s, X^m(\kappa_m(s)),\L(X^m(\kappa_m(s)))) \dd B_s\Big|^\frac{p}{2}\Big]\\
    & \leq C_{p,\sigma} \bE\Big[\Big(\int_t^{t'} (1+ |X^m(\kappa_m(s))|)^2 \dd s \Big)^\frac{p}{2}\Big]\\
    & \lesssim C_{p,\sigma} \bE\Big[\Big(\int_t^{t'} 1+ |X^m(\kappa_m(s))|^2 \dd s \Big)^\frac{p}{2}\Big]\\
    & = C_{p,\sigma} \bE\Big[\Big((t'-t)+\int_t^{t'}|X^m(\kappa_m(s))|^2 \dd s \Big)^\frac{p}{2}\Big]\\
    & \lesssim C_{p,\sigma} \Big((t'-t))^\frac{p}{2}+\bE\Big[\Big(\int_t^{t'}|X^m(\kappa_m(s))|^2 \dd s \Big)^\frac{p}{2}\Big]\Big).
  \end{align*}
  By applying the Jensen inequality, we get
  \begin{align}\label{eq:tightness_M}
    \begin{split}\bE[|M^m(t')-M^m(t)|^p]	& \leq C_{p,\sigma} \Big((t'-t)^\frac{p}{2}+ (t'-t)^{\frac{p}{2}-1} \bE\Big[\int_t^{t'}|X^m(\kappa_m(s))|^p \dd s\Big]\Big)\\
    & \leq C_{p,\sigma} \Big((t'-t)^\frac{p}{2}+ (t'-t)^{\frac{p}{2}-1} \int_t^{t'}\sup_{r \in [0,T]} \bE[|X^m(\kappa_m(r))|^p] \dd s\Big)\\
    & \leq C_{p,\sigma} ((t'-t)^\frac{p}{2}+ (t'-t)^\frac{p}{2} C)\\
    & \lesssim C(t'-t)^{\gamma p}.
    \end{split}
  \end{align}
  Applying the Kolmogorov tightness criterion yields the tightness of the sequence of measures $(\bP_{X^m,A^m,M^m,B})_{m \in \bN}$. Hence, by Prokhorov's theorem (\cite[Theorem~2.4.7]{Karatzas1991}) we obtain the relative compactness (\cite[Definition~2.4.6]{Karatzas1991}) of $(\bP_{X^m,A^m,M^m,B})_{m \in \bN}$ in the space of probability measures on $C([0,T]; \bR^d \times \bR^d \times \bR^d \times \bR^m)$, i.e. there is a converging subsequence $(\bP_{X^{m_k},A^{m_k},M^{m_k},B})_{k \in \bN}$ such that
  \begin{equation*}
    \bP_{X^{m_k},A^{m_k},M^{m_k},B} \to \bP_{X,A,M,B} \quad \text{weakly as } k \to \infty
  \end{equation*}
  for some measure $\bP_{X,A,M,B}$.

  By the Skorokhod representation theorem (see e.g. \cite[Theorem~11.7.2]{Dudley2002}) there is a probability space $(\hat{\Omega}, \hat{\mathcal{F}}, \hat{\bP})$ with continuous stochastic processes $(\hat{X}^k)_{k \in \bN}$, $(\hat{A}^k)_{k \in \bN}$, $(\hat{M}^k)_{k \in \bN}$, $(\hat{B}^k)_{k \in \bN}$ and $\tilde{X}$, $\tilde{A}$, $\tilde{M}$, $\hat{B}$ on it, such that
  \begin{equation}\label{eq:equ_in_law}
    (\hat{X}^k,\hat{A}^k,\hat{M}^k,\hat{B}^k) \stackrel{\mathscr{D}}{\sim} (X^{m_k}, A^{m_k}, M^{m_k}, B), \qquad k \in \bN,
  \end{equation}
  and
  \begin{equation*}
    (\hat{X}^k,\hat{A}^k,\hat{M}^k,\hat{B}^k) \to (\tilde{X},\tilde{A},\tilde{M},\hat{B})
  \end{equation*}
  in $C([0,T]; \bR^d \times \bR^d \times \bR^d \times \bR^m)$ as $k \to \infty$ $\hat{\bP}$-a.s. By a version of the Yamada--Watanabe result, see \cite[Theorem~1.5]{Kurtz2014}, $\hat{M}^k_t=\int_0^t \sigma\big(s, \hat{X}^k(\kappa_{m_k}(s)), \L(\hat{X}^k(\kappa_{m_k}(s)))\big) \dd \hat{B}^k_s$, for $t \in [0,T]$ and $k \in \bN$, and the stochastic processes $(\hat{B}_k)_{k \in \bN}$ are Brownian motions as $\hat{B}^k \stackrel{\mathscr{D}}{\sim} B$. Thus, $\hat{M}^k$ is a local $\hat{\bP}$-martingale with quadratic variation
  \begin{equation*}
    \langle \hat{M}^k \rangle_t = \int_0^t (\sigma \sigma^\top)\big(s, \hat{X}^k(\kappa_{m_k}(s)),\L(\hat{X}^k(\kappa_{m_k}(s)))\big) \dd s.
  \end{equation*}

  Next we show that $\lim_{k \to \infty} W_\eta(\L(\hat{X}^k),\L(\tilde{X}))=0$. By \cite[Theorem~5.5]{Carmona2018} for this convergence to hold, we need the weak convergence of $(\L(\hat{X}^k))_{k \in \bN}$ towards $\L(\tilde{X})$, which is already implied by the above, and uniform integrability in the sense of
  \begin{equation*}
    \lim_{r \to \infty} \sup_{k \in \bN} \bE_{\hat{\bP}}\big[\|\hat{X}^k-y\|^\eta_\infty \mathbbm{1}_{\{\|\hat{X}^k-y\|_\infty\geq r \} }\big]=0
  \end{equation*}
  for some $y \in C([0,T];\bR^d)$. For the latter, we show that there is a uniform bound on $\bE_{\hat{\bP}}[\sup_{t \in [0,T]} |\hat{X}^k(t)|^\eta]$. First, note, by \eqref{eq:equ_in_law}, for all $k \in \bN$ that
  \begin{align*}
    \bE_{\hat{\bP}}\Big[ \sup_{t \in [0,T]} | \hat{X}^k (t)|^\eta\Big]
    & \leq C_\eta  \Big(\bE_{\hat{\bP}}[|\hat{X}^k(0)|^\eta] + \bE_{\hat{\bP}}\Big[ \sup_{t \in [0,T]} |\hat{X}^k(t)-\hat{X}^k(0)|^\eta\Big]\Big)\\
    & \leq C_\eta  \Big(\bE[|X_0|^\eta] + \bE_{\hat{\bP}}\Big[\sup_{0 \leq s,t \leq T} |\hat{X}^k(t)-\hat{X}^k(s)|^\eta\Big]\Big).
  \end{align*}
  $\hat{X}^k$ is $h$-H\"older for all $h \in (0,\gamma-1/p)$ for all $k \in \bN$ by the Kolmogorov continuity theorem (see \cite[Theorem~2.2.8]{Karatzas1991}) and \eqref{eq:tightness}. By Assumption~\ref{ass:kernel} we can choose some $\beta \in (0,\gamma-1/p)$ and $\alpha > \frac{2}{\beta}$ such that $\eta \alpha \leq p$ and obtain by the Garsia--Rodemich--Rumsey inequality (see \cite[Theorem~2.1.3]{Stroock1979}) that
  \begin{align*}
    &\sup_{0 \leq s,t \leq T} |\hat{X}^k(t) - \hat{X}^k(s)|\\
    &\quad \leq 8 \int_0^T \Bigg(\frac{4}{u^2} \int_0^T \int_0^T \Big(\frac{|\hat{X}^k(t)-\hat{X}^k(s)|}{|t-s|^\beta}\Big)^\alpha \dd s \dd t\, \Bigg)^{\frac{1}{\alpha}} \dd u^\beta\\
    &\quad = 8 ~ 4^\frac{1}{\alpha} \Bigg( \int_0^T \frac{1}{u^\frac{2}{\alpha}} \dd u^\beta \Bigg) \Bigg(\int_0^T \int_0^T \Bigg(\frac{|\hat{X}^k(t)-\hat{X}^k(s)|}{|t-s|^\beta}\Bigg)^\alpha \dd s \dd t\Bigg)^\frac{1}{\alpha}\\
    &\quad = 8 ~ 4^\frac{1}{\alpha} \Bigg( \int_0^T \beta u^{\beta-1-\frac{2}{\alpha}} \dd u \Bigg) \Bigg(\int_0^T \int_0^T \Big(\frac{|\hat{X}^k(t)-\hat{X}^k(s)|}{|t-s|^\beta}\Big)^\alpha \dd s \dd t\Bigg)^\frac{1}{\alpha}.
  \end{align*}
  The former integral is finite, since $\beta-1-\frac{2}{\alpha}>-1$. Recall that in \eqref{eq:tightness} we have shown that there is a $C > 0$ such that for all $k \in \bN$
  \begin{equation*}
    \bE_{\hat{\bP}}[|\hat{X}^k(t)-\hat{X}^k(s)|^{\eta \alpha}] \leq C |t-s|^{\eta \alpha \gamma}, \quad t,s \in [0,T].
  \end{equation*}
  Hence, by the Jensen inequality and the linearity of integrals, we have
  \begin{align*}
    \bE_{\hat{\bP}}\big[\sup_{0 \leq s,t \leq T} |\hat{X}^k(t)-\hat{X}^k(s)|^\eta\big] & \leq C \bE_{\hat{\bP}}\Big[\Big(\int_0^T \int_0^T \frac{|\hat{X}^k(t)-\hat{X}^k(s)|^\alpha}{|t-s|^{\alpha \beta}} \dd s \dd t\Big)^\frac{\eta}{\alpha}\Big]\\
    & \leq C T^{\frac{2(\eta-1)}{\alpha}}\bE_{\hat{\bP}}\Big[\Big(\int_0^T \int_0^T \frac{|\hat{X}^k(t)-\hat{X}^k(s)|^{\eta \alpha}}{|t-s|^{\eta \alpha \beta}} \dd s \dd t\Big)^\frac{1}{\alpha}\Big]\\
    & \leq C_{T,\eta,\alpha} \Big(\int_0^T \int_0^T \frac{\bE_{\hat{\bP}}[|\hat{X}^k(t)-\hat{X}^k(s)|^{\eta \alpha}]}{|t-s|^{\eta \alpha \beta}} \dd s \dd t\Big)^\frac{1}{\alpha}\\
    & \leq C_{T,\eta,\alpha} \Big(\int_0^T \int_0^T \frac{C|t-s|^{\eta \gamma \alpha}}{|t-s|^{\eta \alpha \beta}} \dd s \dd t\Big)^\frac{1}{\alpha}\\
    & \leq \tilde{C} \Big(\int_0^T \int_0^T |t-s|^{\eta \alpha(\gamma-\beta)} \dd s \dd t\Big)^\frac{1}{\alpha}.
  \end{align*}
  Note that $\eta \alpha (\gamma-\beta) > 0$, hence the integral is finite. This shows that there is a $C > 0$ such that for all $k \in \bN$ we have $\bE_{\hat{\bP}}[ \sup_{t \in [0,T]} | \hat{X}^k (t)|^\eta] \leq C$, and the convergence of the Wasserstein distance of the laws follows.

  Due to the $\hat{\P}$-a.s. convergence of the sequence of martingales $(\hat{M}^k)_{k\in \N}$ to $\tilde{M}$, \cite[Proposition~IX.1.17]{Jacod2003} implies that $\tilde{M}$ is also a local $\hat\P$-martingale, and the convergence of
  \begin{equation*}
    \int_0^t (\sigma \sigma^\top)\big(s, \hat{X}^k(\kappa_{m_k}(s)),\L(\hat{X}^k(\kappa_{m_k}(s)))\big) \dd s
  \end{equation*}
  in probability, see Lemma~\ref{lemma:convergence_integrals}, together with \cite[Corollary~VI.6.29]{Jacod2003} implies that the quadratic variation of $\tilde{M}$ is $\langle \tilde{M}\rangle_t=\int_0^t (\sigma \sigma^\top)(s, \tilde{X}_s,\L(\tilde{X}_s)) \dd s$. Therefore, the representation theorem for local martingales with absolutely continuous quadratic variations (see e.g. \cite[Theorem~II.7.1$^\prime$]{Ikeda_Watanabe1989}) yields the existence of some probability space $(\tilde{\Omega},\tilde{\mathcal{F}},\tilde{\P})$, which is an extension of $(\hat{\Omega},\hat{\mathcal{F}},\hat{P})$, and a Brownian motion $(\tilde{B}_t)_{t\in[0,T]}$ on it, such that $\tilde{M}_t=\int_0^t \sigma(s,\tilde{X}_s,\L(\tilde{X}_s))\dd \tilde{B}_s$ for $t\in[0,T]$.
\end{proof}

As third step, we derive a solution to the Volterra local martingale problem, based on the converging sequences obtained in Lemma~\ref{lem:limitprocesses}.

\begin{lemma}\label{lem:solution to Volterra martingale problem}
  Suppose Assumption~\ref{ass:volatility} and~\ref{ass:kernel}. Then, there exists a solution $P$ to the Volterra local martingale problem given $(\mu_0,b,\sigma,K_b,K_\sigma)$.
\end{lemma}

\begin{proof}
  We introduce the stochastic processes $(Z^m)_{m\in\N}$, $(\hat{Z}^k)_{k\in\N}$ and $Z$ by
  \begin{equation*}
    Z_t^m := A_t^m + M_t^m,\quad \hat{Z}_t^k:=\hat{A}_t^k + \hat{M}_t^k \quad\text{and}\quad \tilde{Z}_t:= \tilde{A}_t + \tilde{M}_t, \quad t\in[0,T].
  \end{equation*}
  on the probability space $(\tilde{\Omega}, \tilde{\mathcal{F}},\tilde{\bP})$, as defined in Lemma~\ref{lem:limitprocesses}.
  
  Consider the measure $P$ on $(C([0,T];\bR^d),\mathcal{B}(C([0,T];\bR^d)))$ defined by $P = \tilde{\bP} \circ \tilde{Z}^{-1}$ with the filtration $(\mathcal{G}_t)_{t \in [0,T]}$ being the right-continuous augmentation of the completion of $\sigma(\tilde{Z}_s, 0 \leq s \leq t)$. We show that $P$ solves the distribution-dependent Volterra local martingale problem given $(\mu_0,b,\sigma,K_b,K_\sigma)$ in the sense of Definition~\ref{def:martproblem_new}. It is clear by construction that $\tilde{Z}$ is a semimartingale. It remains to show the two equations in (ii) of Definition~\ref{def:martproblem_new}:
	
  We start with \eqref{def:M_f_new}. For $k\in\N$ and $f\in C_0^2(\R^d)$, the stochastic process $(\mathcal{M}^{f,k}_t)_{t\in[0,T]}$ is defined by
  \begin{equation*}
    \mathcal{M}^{f,k}_t := f(\hat{Z}_t^k)-\int_0^t \mathcal{A}^{f}\big(s,\hat{X}_k(\kappa_{m_k}(s)),\L(\hat{X}^k(\kappa_{m_k}(s))),\hat{Z}_s^k\big)\dd s,\quad t\in [0,T].
  \end{equation*}
  Due to $(\hat{X}^k,\hat{Z}^k)\stackrel{\mathscr{D}}{\sim}(X^{n_k},Z^{n_k})$ and applying It{\^o} formula on $f(\hat{Z}^k_t)$, we have
  \begin{align*}
    \mathcal{M}^{f,k}_t &= f(\hat{Z}_t^k)-\int_0^t \mathcal{A}^{f}\big(s,\hat{X}^k(\kappa_{m_k}(s)),\L(\hat{X}^k(\kappa_{m_k}(s))),\hat{Z}_s^k\big)\dd s\\
    & = f(\hat{Z}_t^k)- \int_0^t \sum_{i=1}^d b_i\big(s, \hat{X}_k(\kappa_{m_k}(s)),\L(\hat{X}^k(\kappa_{m_k}(s)))\big) \frac{\partial f(\hat{Z}_s^k)}{\partial z_i} \dd s \\
    & \qquad - \int_0^t \frac{1}{2} \sum_{i,j=1}^d (\sigma \sigma^\top)_{ij} \big(s, \hat{X}^k(\kappa_{m_k}(s)),\L(\hat{X}^k(\kappa_{m_k}(s)))\big) \frac{\partial^2 f(\hat{Z}_s^k)}{\partial z_i \partial z_j} \dd s\\
    & = f(\hat{Z}_0^k) + \sum_{j=1}^m \sum_{i=1}^d \int_0^t \frac{\partial f(\hat{Z}_s^k)}{\partial z_i} \sigma_{ij} \big(s, \hat{X}^k(\kappa_{m_k}(s)),\L(\hat{X}^k(\kappa_{m_k}(s)))\big)  \dd \tilde{B}_s^j.
  \end{align*}
  Note that $\hat{Z}^k \to \tilde{Z}$ by construction as $k \to \infty$. Now Lemma~\ref{lemma:convergence_integrals} implies
  \begin{align*}
    \int_0^t \mathcal{A}^{f}\big(s,\hat{X}^k(\kappa_{m_k}(s)),\L(\hat{X}^k(\kappa_{m_k}(s))),\hat{Z}_s^k\big)\dd s \to \int_0^t \mathcal{A}^{f}(s,\tilde{X}_s,\L(\tilde{X}_s),\tilde{Z}_s)\dd s
  \end{align*}
  and thus $\mathcal{M}^{f,k} \to \mathcal{M}^{f}$ weakly as $k \to \infty$. Then, the limiting process $(\mathcal{M}^{f}_t)_{t \in [0,T]}$ is itself a local martingale on $(C([0,T];\bR^d),\mathcal{B}(C([0,T];\bR^d)), (\mathcal{G}_t)_{t\in[0,T]},P)$ by \cite[Proposition~IX.1.17]{Jacod2003}.

  Next we verify \eqref{eq:X_MP_new}. We obtain by \cite[Theorem~1.1]{Wang2008} pathwise uniqueness of $\hat{X}^k$, $k \in \bN$, iteratively on $[t_i^{m_k},t_{i+1}^{m_k}]$, $i = 0,1,\ldots, N(m_k)-1$. To see this, note that the coefficients only change in time on these intervals. Then, the general version of the Yamada--Watanabe result (\cite[Theorem~1.5]{Kurtz2014}) yields that $\hat{X}^k$ can be represented as the stochastic output of the distribution-dependent stochastic Volterra equation \eqref{eq:MVSVE} from the stochastic input $\hat{M}^k$ in the same way as $X^{m_k}$ from $M^{m_k}$, hence, we get that
  \begin{equation}\label{eq:sve_skorokhod}
    \hat{X}^k(t)
    = \hat{X}^k_0+\int_0^t K_b(s,t)\dd \hat{A}^k_s+\int_0^t K_\sigma(s,t)\dd \hat{M}^k_s,\quad t\in[0,T],
  \end{equation} holds $P$-a.s. Note that $\tilde{Z}_0=\tilde{A}_0=\hat{A}^k_0=\hat{X}^k_0=\tilde{X}_0 \sim \mu$ for all $k \in \bN$. By Lemma~\ref{lemma:convergence_integrals} we have
  \begin{equation*}
    \Big(\int_0^t K_b(s,t)\dd \hat{A}^k_s \Big)_{t \in [0,T]} \xrightarrow{P} \Big(\int_0^t K_b(s,t)\dd A_s\Big)_{t \in [0,T]}
  \end{equation*}
  in $C([0,T];\R^d)$ as $k\to\infty$. For the martingale part we can prove uniform integrability of $((\hat{M}^k)^2)_{k \in \bN}$ following the same steps as we used to show the uniform integrability of $((\hat{X}^k)^\eta)_{k \in \bN}$ near the end of the proof of Lemma~\ref{lem:limitprocesses}. When doing so in place of \eqref{eq:tightness} one can use \eqref{eq:tightness_M}. Then, by the Vitali convergence theorem, see \cite[Exercise~4.13]{Cinlar2011}, $(\hat{M}^k)_{k \in \bN}$ converges towards $\tilde{M}$ in $\mathcal{H}^2$. By the It{\^o} isometry, for any $ t \in [0,T]$, we have
  \begin{align*}
    \bE_P\Big[ \Big( \int_0^t K_\sigma (s,t) \dd (\hat{M}^k_s - \tilde{M}_s) \Big)^2 \Big] = \bE_P\Big[\int_0^t (K_\sigma(s,t))^2 \dd \langle \hat{M}^k-\tilde{M}\rangle_s\Big].
  \end{align*}
  By the $\mathcal{H}^2$-convergence of $(\hat{M}^k)_{k \in \bN}$, we have $\langle \hat{M}^k-\tilde{M}\rangle_T \to 0$ and since $t \mapsto \langle \hat{M}^k-\tilde{M}\rangle_t$ is increasing we can conclude that
  \begin{align*}
    \bE_P\Big[ \Big( \int_0^t K_\sigma(s,t) \dd (\hat{M}^k_s - \tilde{M}_s)\Big)^2 \Big] \to 0, \quad t \in [0,T].
  \end{align*}
  This implies
  \begin{equation*}
    \int_0^t K_\sigma(s,t) \dd \hat{M}_s^k \xrightarrow{P} \int_0^t K_\sigma(s,t) \dd \tilde{M}_s
  \end{equation*}
  for all $t \in [0,T]$. We know that there is a continuous stochastic process $\tilde{N}=(\tilde{N}_t)_{t \in [0,T]}$ such that
  \begin{equation*}
    \Big(\int_0^t K_\sigma(s,t) \dd \hat{M}_s^k\Big)_{t \in [0,T]} \xrightarrow{P} \tilde{N},
  \end{equation*}
  i.e. we have uniform convergence in probability. In particular, for each fixed $t$, we have
  \begin{equation*}
    \Big|\int_0^t K_\sigma(s,t) \dd \hat{M}_s^k - \tilde{N}_t\Big| \leq \sup_{r \in [0,T]} \Big|\int_0^r K_\sigma(s,r) \dd \hat{M}_s^k - \tilde{N}_r\Big| \xrightarrow{P} 0.
  \end{equation*}
  Therefore, we observe that
  \begin{equation*}
    \int_0^t K_\sigma(s,t) \dd \hat{M}_s^k \xrightarrow{P} \tilde{N}_t
    \quad\text{and}\quad
    \int_0^t K_\sigma(s,t) \dd \hat{M}_s^k \xrightarrow{P} \int_0^t K_\sigma(s,t) \dd \tilde{M}_s.
  \end{equation*}
  Since the limit is unique, we have
  \begin{equation*}
    \tilde{N}_t = \int_0^t K_\sigma(s,t) \dd \tilde{M}_s, \qquad P\text{-a.s.}
  \end{equation*}
  Since this holds for each $t \in [0,T]$, we have that
  \begin{equation*}
    P \Big(\tilde{N}_t = \int_0^t K_\sigma(s,t) \dd \tilde{M}_s\Big) = 1 \quad\text{for all } t \in [0,T],
  \end{equation*}
  i.e. the processes are modifications of each other. Continuous stochastic processes which are modifications are indistinguishable. Hence,
  \begin{equation*}
    P \Big(\tilde{N}_t = \int_0^t K_\sigma(s,t) \dd \tilde{M}_s \text{ for all } t \in [0,T]\Big) = 1,
  \end{equation*}
  i.e., $\tilde{N} = \int_0^\cdot K_\sigma(s,\cdot) \dd \tilde{M}_s$ almost surely in $C([0,T];\bR^d)$.
   
  We can take the limit in probability in \eqref{eq:sve_skorokhod} or the $P$-a.s. limit for some subsequence, to obtain that \eqref{eq:X_MP_new} holds $P$-a.s.
\end{proof}

\section{Properties of solutions to distribution-dependent SVEs}\label{sec:rest}

In this section we derive some results regarding the integrability and sample path regularity of (weak) solutions to the distribution-dependent stochastic Volterra equation~\eqref{eq:MVSVE}. We start by showing that solutions to distribution-dependent SVEs are uniformly bounded in~$L^p$.

\begin{lemma}
  Suppose Assumption~\ref{ass:volatility} and~\ref{ass:kernel}. Given the distribution-dependent SVE~\eqref{eq:MVSVE} there is a $C_p >0$ such that $\bE[|X(t)|^p] \leq C_p$ holds for any $t \in [0,T]$.
\end{lemma}

\begin{proof}
  Take the approximating sequence $(\hat{X}^k)_{k \in \bN}$ from the previous section. $(\hat{X}^k)_{k \in \bN}$ converges strongly against $X$ (see proof of Lemma~\ref{lem:limitprocesses}). By \eqref{eq:equ_in_law} we have $\hat{X}^k  \stackrel{\mathscr{D}}{\sim} X^{m_k}$, $k \in \bN$, for some increasing sequence $(m_k)_{k \in \bN} \subset \bN$ and it follows with Lemma~\ref{lemma:boundedness} that $\bE[|\hat{X}^k(t)|^p] \leq C_p$ for all $k \in \bN$ and $t \in [0,T]$ and some $C_p > 0$. Then, by Fatou's Lemma we have
  \begin{align*}
    \bE[|X_t|^p] \leq \liminf_{k \to \infty} \bE [|\hat{X}^k(t)|^p] \leq C_p.
  \end{align*}
\end{proof}

Next, we turn to the regularity of the solutions. In general, solutions to stochastic Volterra equations need not be continuous. As we will see the regularity is strongly linked to the regularity of the kernel. For more on the path regularity of standard stochastic Volterra equations of convolution type we refer the reader to \cite{AbiJaber2021}.

\begin{lemma}
  Suppose Assumption~\ref{ass:volatility} and~\ref{ass:kernel}. Let $X$ be a solution to the distribution-dependent stochastic Volterra equation \eqref{eq:MVSVE}. Then there exists a modification $X^\ast$ of $X$  that is $h$-H{\"o}lder-continuous for any $h \in (0,\frac{\gamma p-1}{p})$ with $\gamma$ and $p$ as in Assumption~\ref{ass:kernel}.
\end{lemma}

\begin{proof}
  In the proof of Lemma~\ref{lem:limitprocesses} we have seen that there is a $C>0$ such that for any $X^m$ of the approximating sequence $(X^m)_{m \in \bN}$ we have $\bE[|X^m(t')-X^m(t)|^p] \leq C(t'-t)^{\gamma p}$ for all $(t,t')\in \Delta_T$. By the same argumentation as in the previous proof using Fatou's Lemma the same holds true for $X$. Then by the Kolmogorov continuity theorem (also known as Kolmogorov-Tschenzow, see \cite[Theorem~2.2.8]{Karatzas1991}) there is a H{\"o}lder continuous modification of order $h$ for all $h \in (0,\frac{\gamma p-1}{p})$.
\end{proof}

\appendix
\section{Convergence result for Volterra-type integrals}\label{sec: appendix}

This appendix contains an auxiliary result regarding the convergence in probability of Volterra-type integrals.

\begin{lemma}\label{lemma:convergence_integrals}
  Let $f \colon [0,T] \times \bR^d \times \mathcal{P}_\eta(\bR^d) \to \bR^e$ be a function such that for all compact sets $\mathcal K \subset \bR^d$ and $\bar{\mathcal K} \subset \mathcal{P}_\eta(\bR^d)$ and every $\epsilon > 0$ there exists a $\delta >0$ such that
  \begin{equation*}
    |f(t,x,\mu)-f(t,y,\nu)|\leq \epsilon,
  \end{equation*}
  for all $t \in [0,T]$, $x,y \in \mathcal K$ satisfying $|x-y| \leq \delta$, and $\mu, \nu \in \bar{\mathcal K}$ with $W_\eta(\mu,\nu) \leq \delta$, and such that $f$ fulfills the growth condition
  \begin{align*}
    |f(t,x,\mu)| \leq C(1+|x| + W_\eta(\delta_0,\mu))^2, \quad t \in [0,T], ~x \in \bR^d, ~\mu \in \mathcal{P}_\eta(\bR^d).
  \end{align*}
  Let $K \colon \Delta_T \to \bR$ be measurable and bounded function in $L^1([0,T])$ uniformly in the second variable, i.e. $\sup_{t \in [0,T]} \int_0^t |K(s,t)| \dd s \leq M$ for some $M \geq 0$. Let $(\pi_k = \{t_0^k=0, \ldots, t_{N(k)}^k=T\})_{k \in \bN}$ be a sequence of partitions with $\lim_{k \to \infty} |\pi_k|=0$ and define, for every $k \in \bN$, the function $\kappa_k \colon [0,T] \to [0,T]$ by
  \begin{equation*}
    \kappa_k(T) := T \text{ and } \kappa_k(t) := t_i^k \text{ if } t_i^k \leq t < t_{i+1}^k,\quad \text{for all } i=0,1,\ldots,N(k)-1.
  \end{equation*}
  If $(X^k)_{k \in \bN}$ is a sequence of continuous stochastic processes with uniformly finite $\eta$-th moments such that $X^k \to X$ in $C([0,T];\bR^d)$ as $k \to \infty$ $\bP$-a.s., $W_\eta(\L(X^k),\L(X))\to 0$ as $k \to \infty$ and
  \begin{align}\label{eq:Tight_Crit_1}
    \bE[|X^k(s)-X^k(t)|^\eta] \to 0
  \end{align}
  uniformly in $k$ as $|s-t|\to 0$, $s,t \in [0,T]$, then one has that
  \begin{equation*}
    \int_0^{\cdot} K(s,\cdot) f\big(s, X^k(\kappa_k(s)),\L(X^k(\kappa_k(s)))\big)\dd s \xrightarrow{\bP} \int_0^{\cdot} K(s,\cdot) f\big(s, X_s,\L(X_s)\big)\dd s
  \end{equation*}
  with respect to $\|\cdot\|_\infty$ as $k \to \infty$, where $\xrightarrow{\bP}$ denotes convergence in probability.
\end{lemma}

\begin{proof}
  Let $\epsilon, \theta>0$ be arbitrary but fixed. Choose $N, K_1 \in \bN$ such that
  \begin{equation*}
    \bP\Big(\|X\|_\infty \geq \frac{N}{2}\Big) \leq \frac{\theta}{4}
    \quad\text{and}\quad
    \bP\Big(\|X^k-X\|_\infty \geq \frac{N}{2}\Big) \leq \frac{\theta}{4},
  \end{equation*}
  for all $k \geq K_1$. Then, we get
  \begin{align*}
    \bP(\|X\|_\infty \vee \|X^k\|_\infty \geq N) & \leq \bP(\{\|X\|_\infty \geq N \} \cup \{\|X^k-X\|_\infty + \|X\|_\infty \geq N\})\\
    & \leq \bP \Big(\| X \|_\infty \geq \frac{N}{2}\Big) + \bP \Big(\|X^k - X \|_\infty \geq \frac{N}{2}\Big)\\
    & \leq \frac{\theta}{2}.
  \end{align*}
  For every $k \in \bN$ and $t \in [0,T]$, we have
  \begin{align*}
    &|A^k_t - A_t| \\
    &\quad:= \Big|\int_0^t K(s,t) f\big(s,X^k(\kappa_k(s)),\L(X^k(\kappa_k(s)))\big)\dd s - \int_0^t K(s,t) f\big(s,X(s),\L(X(s))\big)\dd s \Big|\\
    & \quad\leq M \Big( \sup_{s \in [0,T]} |f\big(s,X^k(\kappa_k(s)), \L(X^k(\kappa_k(s)))\big) - f\big(s,X^k(s), \L(X^k(\kappa_k(s)))\big)|\\
    &\quad \qquad + \sup_{ s \in [0,T]} | f\big(s,X^k(s), \L(X^k(\kappa_k(s)))\big) - f\big(s,X^k(s), \L(X^k(s))\big)|\\
    &\quad \qquad + \sup_{ s \in [0,T]} |f\big(s,X^k(s), \L(X^k(s))\big) - f\big(s,X_s, \L(X_s)\big)|\Big).
  \end{align*}
  By assumption there is some compact $\bar{\mathcal K} \subset \mathcal P_\eta(\bR^d)$ such that $\mathcal L(X_t^k) \in \bar{\mathcal K}$ for all $k \in \bN$ and $t \in [0,T]$. Let $\delta >0$ be such that
  \begin{equation*}
    |f(t,x,\mu)-f(t,y,\nu)|\leq \frac{\epsilon}{3M},
  \end{equation*}
  for all $t \in [0,T]$, $x,y \in [-N,N]$ satisfying $|x-y| \leq \delta$ and $\mu, \nu \in \bar{\mathcal K}$ with $W_\eta(\mu,\nu) \leq \delta$.

  The convergence of the sequence $(X^k)_{k \in \bN}$ of stochastic processes implies tightness of their laws. By this tightness and by \cite[Theorem~2.4.10]{Karatzas1991}, there is a $K_2 \in \bN$ such that for all $ k \geq K_2$ we have
  \begin{equation*}
    \bP\Big(\sup_{s \in [0,T]} | X^k(\kappa_k(s))-X^k(s)| > \delta \Big) \leq \frac{\theta}{4}.
  \end{equation*}
  Note that on $\{\|X\|_\infty \vee \|X^k\|_\infty \leq N\} \cap \{\sup_{s \in [0,T]} |X^k(\kappa_k(s))-X^k(s)|\leq \delta \}$ we have
  \begin{equation*}
    \sup_{s \in [0,T]} \big|f(s,X^k(\kappa_k(s)), \L(X^k(\kappa_k(s)))) -f(s,X^k(s),\L(X^k(\kappa_k(s))))\big| \leq \frac{\epsilon}{3M}.
  \end{equation*}

  By \eqref{eq:Tight_Crit_1}, by $|\pi_k| \to 0$ as $k\to \infty$ and by
  \begin{equation*}
    W_\eta(\L(X^k(\kappa_k(s))),\L(X^k(s))) \leq \bE[|X^k(\kappa_k(s))-X^k(s)|^\eta]^{\frac{1}{\eta}},
  \end{equation*}
  we can find a $K_3 \in \bN$ such that, for all $k \geq K_3$, we have $W_\eta(\L(X^k(\kappa_k(s))),\L(X^k(s)))\leq \delta$ for all $s \in [0,T]$. Hence, we obtain
  \begin{equation*}
    \sup_{s \in [0,T]} \big|f(s,X^k(s),\L(X^k(\kappa_k(s))))-f(s,X^k(s),\L(X^k(s)))\big| \leq \frac{\epsilon}{3M}
  \end{equation*}
  for all $k \geq K_3$ on $\{\|X\|_\infty \vee \|X^k\|_\infty \leq N\}$.

  By $\lim_{k \to \infty} W_\eta(\L(X^k),\L(X))=0$ together with Remark~\ref{remark_Wasserstein}, there is a $K_4 \in \bN$ such that we have the bound
  \begin{equation*}
    \sup_{t \in [0,T]} W_\eta(\L(X^k(t)),\L(X_t))\leq \delta
  \end{equation*}
  for all $k \geq K_4$. By the $\bP$-a.s.-convergence of $(X^k)_{k \in \bN}$ to $X$ there is a $K_5 \in \bN$, such that $\bP(\|X^k-X\|_\infty > \delta) \leq \frac{\theta}{4}$. Hence, for all $k \geq K_4$ on $\{\|X^k-X\|_\infty \leq \delta\} \cap \{\|X\|_\infty \vee \|X^k\|_\infty \leq N\}$ we have
  \begin{equation*}
    \sup_{ s \in [0,T]} \big|f(s,X^k(s), \L(X^k(s))) - f(s,X_s, \L(X_s))\big| \leq \frac{\epsilon}{3M}.
  \end{equation*}

  Setting $K := \max_i\{K_i\}$, we get
  \begin{align*}
    &\bP(\|A^k-A\|_\infty \geq \epsilon) \\
    &\quad \leq \bP(\{\|A^k-A\|_\infty \geq \epsilon\} \cap \{\|X\|_\infty \vee \|X^k\|_\infty \leq N\}) + \bP(\|X\|_\infty \vee \|X^k\|_\infty\geq N)\\
    &\quad \leq \bP(\{\|X^k - X \|_\infty > \delta \} \cap \{ \|X\|_\infty \vee \|X^k\|_\infty \leq N \} )\\
    & \quad + \bP(\{\sup_{s \in [0,T]} |X^k(\kappa_k(s))-X^k(s)| > \delta \} \cap \{ \|X\|_\infty \vee \|X^k\|_\infty \leq N\})\\
    &\quad \quad + \bP(\|X\|_\infty \vee \|X^k\|_\infty \geq N)\\
    &\quad \leq \frac{\theta}{4} + \frac{\theta}{4} + \frac{\theta}{2}=\theta
  \end{align*}
  for all $k \geq K$.
\end{proof}

\bibliography{literature}{}
\bibliographystyle{amsalpha}

\end{document}